\documentclass[11pt]{amsart}
\usepackage[utf8]{inputenc}
\usepackage{amssymb,amsmath, amsthm, mathabx, mathtools}
\usepackage{color, enumerate}
\usepackage{hyperref}
\usepackage{caption}
\usepackage{subcaption}
\usepackage{a4wide}
\usepackage{accents}
\usepackage{tikz}
\usepackage{graphicx}
\usepackage[percent]{overpic}
\usepackage{chngcntr}
\usepackage{datetime2}
\usepackage{scalerel}
\usepackage{bm}
\usepackage{tcolorbox}
\usepackage{xcolor}

\newtheorem{thm}{Theorem}[section]
\newtheorem{prop}[thm]{Proposition}
\newtheorem{lem}[thm]{Lemma}
\newtheorem{cor}[thm]{Corollary}
\theoremstyle{definition}

\newcommand{\bbE}{\mathbb{E}}

\newcommand{\bbP}{\mathbb{P}}

\newcommand{\bbR}{\mathbb{R}}

\newcommand{\bbZ}{\mathbb{Z}}

\newcommand{\bfP}{\mathbf{P}}

\newcommand{\bfu}{\mathbf{u}}
\newcommand{\bfv}{\mathbf{v}}

\allowdisplaybreaks

\newcommand{\one}{\mathbf{1}}

\newcommand{\mrJ}{\mathrm{J}}

\newcommand{\mrL}{\mathrm{L}}

\newcommand{\mrM}{\mathrm{M}}

\newcommand{\Ber}{\mathrm{Ber}}

\newcommand{\Pois}{\mathrm{Pois}}

\newcommand{\mrI}{\mathrm{I}}

\title{Local convergence in $t$-PNG}

\author{M\'arton Bal\'azs}
\address[M\'arton Bal\'azs]{School of Mathematics, University of Bristol.}
\email{m.balazs@bristol.ac.uk}

\author{Ruby Bestwick}
\address[Ruby Bestwick]{Mathematical Institute, University of Oxford. Part of this work was done while Ruby was affiliated with the University of Bristol.}
\email{ruby.bestwick@bnc.ox.ac.uk}

\author{Artem Borisov}
\address[Artem Borisov]{Department of Mathematics, Massachusetts Institute of Technology. This work was done while Artem was affiliated with the University of Bristol.}
\email{artembor@mit.edu}

\author{Elnur Emrah}
\address[Elnur Emrah]{School of Mathematics and Statistics, University College Dublin. Part of this work was done while Elnur was affiliated with the University of Bristol.}
\email{elnur.emrah@ucd.ie}

\author{Jessica Jay} 
\address[Jessica Jay] {School of Mathematical Sciences, Lancaster University.}
\email{j.jay@lancaster.ac.uk}

\begin{document}

\begin{abstract}
 We prove local convergence of the $t$-PNG model with zero boundary to the stationary $t$-PNG model, confirming a recent conjecture of Drillick and Lin \cite[Remark 1.4]{Dril_Lin_24}. The stationary $t$-PNG model is the one with both left and bottom boundaries of Poisson nucleations with rate parameters $\frac{1}{\lambda(1-t)}$ and $\lambda$, respectively, for some $\lambda>0$. In the proof, we consider the trajectories of certain second class particles via an increasing coupling of three $t$-PNG processes, and adapt \emph{microscopic concavity} ideas used in particle models (e.g., Bal\'azs--Sepp\"al\"ainen \cite{Bala_Sepp_09}), as well as blocking measure bounds like in Ferrari--Kipnis--Saada \cite{fks_91}.
\end{abstract}
\maketitle

\section{Introduction}

\subsection{Background.}
A popular class of models for randomly growing interfaces are polynuclear growth models (PNG) 
\cite{meakin_98}. 
The growing interface is represented by a continuous broken line in the Euclidean plane, made up of horizontal linear segments and up/down steps of height 1 between them. Over time these up steps move at speed 1 to the left and similarly the down steps at speed 1 to the right, respectively. These moving up/down steps can be thought of as the growth of islands. When an up step and a down step meet they annihilate each other; in other words, the corresponding islands merge. Also new islands are created randomly by adding a new up and down step pair at infinitesimal distance from each other that immediately begin to grow. The creation of new islands, also known as \emph{nucleations}, happens at the space-time points of a 2-dimensional Poisson process of intensity 1.

In this paper we consider the $t$-PNG model recently introduced by Aggarwal, Borodin and Wheeler \cite{Agga_Boro_Whee_23}. The $t$-PNG model is a one-parameter deformation of PNG where, upon a down step from the left meeting an up step from the right, with probability \(1-t\) they annihilate each other leaving a horizontal segment behind, while with probability \(t\) they send off a new up step to the left and down step to the right, creating a new growing island this way. The original PNG model is recovered as the case \(t=0\).

Since their introduction, PNG models have been well studied. It is well known that they belong to the Kardar--Parisi--Zhang (KPZ) universality class (see e.g., the work of Krug and Spohn \cite{Krug_Spoh_92}). As such, links between PNG and other stochastic processes have also been explored. For example in \cite{Joha_03}, Johansson considered the discrete polynuclear growth model, a special case of this is closely related to last-passage percolation as in Johansson \cite{Joha_00}. Johansson built on the work of Pr\"ahofer and Spohn \cite{Prah_Spoh_02} to prove a functional limit theorem of the convergence of discrete PNG to the Airy process. 

More recently, Fitzgerald \cite{Fitz_24} considered a discrete time random interface growth model which, in a particular limit, converges to PNG. In particular, the growth model considered is defined similarly to PNG but where the boundaries of an island (started from a nucleation) spread stochastically as opposed to at deterministic speed. Fitzgerald remarks that PNG can be viewed as one model in a larger class of exactly solvable growth models, in the way that the height function of totally asymmetric simple exclusion (TASEP) can be seen as one in the larger class of growth models including the height functions of asymmetric simple exclusion (ASEP), $q$-TASEP and PushASEP. The PNG model with stochastic spread can be interpreted via interacting particle systems with blocking and pushing interactions. 
 
In \cite{Agga_Boro_Whee_23}, the authors proved a one-point fluctuation result for $t$-PNG, demonstrating that this deformed version still belongs to the KPZ universality class. More precisely, they showed that the height function associated to \(t\)-PNG, which we also define later, has Tracy-Widom fluctuation limits under the usual scaling of the KPZ universality class. They also prove weak convergence of the rescaled height function to the Cole-Hopf solution of the KPZ equation with narrow wedge initial condition. The one-parameter family of these \(t\)-PNG models provide a partial ordering of line ensembles a little bit similar to how ASEP is a one-parameter softening of the totally asymmetric version TASEP. Indeed the additional coupling constructions we had to introduce in this paper have their counterparts in the ASEP literature. Both for Aggarwal, Borodin, Wheeler and for us proving results on this softened version of PNG represented new challenges hence improved the pool of methods that can be used for such stochastic models.

Since its introduction, further results on \(t\)-PNG appeared. Drillick and Lin \cite{Dril_Lin_24} employed softer arguments to prove a.s.\ convergence to the shape function in the law of large numbers regime. Das, Liao and Mucciconi \cite{das_liao_mucciconi_ldp_for_q-def_png} determine the large deviation behaviour of the \(t\)-PNG height function.

The PNG model lends itself to several further interpretations; see, for example, Forrester \cite{Forr_03}. The \emph{droplet} version of PNG has initial interface that is flat and all nucleations are taken to happen in the light cone of the origin, $\{(x,\tau)\in\mathbb{R}\times\mathbb{R}_{\geq 0}: |x|<\tau\}$. Following the space-time trajectories of the up and down steps, then rotating this picture by 45 degrees turns the PNG model into \emph{Hammersley's process}. In fact, Hammersley's original description \cite{Hamm_72} in 1972 well pre-dates PNG. In this picture one considers a unit intensity Poisson process on $\mathbb{R}^2$; these correspond to the nucleation points in PNG. A set of points $(x_i,y_i)$ in the plane form a \emph{chain} if there is an up-right path passing through all the points. The length of a chain is then defined to be the number of points in the chain. Pick vectors \(\mathbf u\) and \(\mathbf v\) in \(\mathbb R_{\ge0}^2\) with $\mathbf0\le\mathbf u<\mathbf v$ understood coordinatewise and take $S_{\mathbf u,\mathbf v}$ to be the rectangle (with sides parallel to the coordinate axes) with opposite corners $\mathbf u$ and $\mathbf v$ in the Euclidean plane. Then consider the points of the Poisson process that lie in $S_{\mathbf u,\mathbf v}$. We set $\mrL_{\mathbf u,\mathbf v}$ to be the length of the longest chain that can be formed from the Poisson points in $S_{\mathbf u,\mathbf v}$. These chain lengths define Hammersley's process. In \cite{Hamm_72} it is shown that if there are $n$ points in $S_{\mathbf0,\mathbf v}$, then $\mrL_{\mathbf0,\mathbf v}$ has the same distribution as the length of the longest non-decreasing subsequence of a uniform permutation of $[n] = \{1, 2, \dotsc, n\}$.

Given the Poisson process, the level sets of the function \(\mathbf v\mapsto\mrL_{\mathbf0,\mathbf v}\) partition the positive quadrant. It is easy to see that the boundaries between these level sets are formed by infinite upwards and rightwards rays emitted from the Poisson points, mutually annihilating further upwards and rightwards from the locations where they get in contact with rays emitted from other Poisson points. We call these contact-and-annihilation points \emph{corners}.

The $t$-PNG model has a very similar interpretation. The only modification needed to the above is that when when two rays meet they either annihilate (as in Hammersley's process) and a corner point is formed with probability $1-t$, or both rays continue with no annihilation and a crossing point is formed instead with probability $t$. Subsequent contacts with further rays emitted from other Poisson points can again result in annihilation and hence the creation of corner points, or crossings with no annihilation. Given the Poisson configuration, the decision of forming a corner vs.\ a crossing point is always made independently each time with respective probabilities \(1-t\) and \(t\).

The droplet case can be generalised to accommodate nucleations outside the light cone $\{|x|<\tau\}$. It is convenient to keep and rotate this cone as before, and use the same Poisson points on \(\mathbb R_{\ge0}^2\) as well. The effect of nucleations outside this cone will manifest as \emph{sources} and \emph{sinks}. This terminology comes from a particle interpretation of the process which we will briefly mention later. Sources and sinks correspond to upwards and rightwards rays entering the south and west boundaries respectively of \(\mathbb R_{\ge0}^2\), and forming corners or crossing points (for \(t\)-PNG) the same way as rays emitted from Poisson points in \(\mathbb R_{\ge0}^2\) do. We also get similar boundary effects if we restrict our view of Hammersley's process in the quadrant to a rectangle $S_{\mathbf u,\mathbf v}$.

Considering these boundary effects in Hammersley's process allowed the construction of a stationary version of the model by Cator and Groeneboom \cite{Cato_Groe_05}. In this case the sources and sinks follow Poisson processes on the south and west boundaries of \(\mathbb R_{\ge0}^2\) with specific intensities. Stationarity occurs in the sense that restricting the view into a smaller box $S_{\mathbf u,\mathbf v}$ will give rise to sources and sinks for this box with the same boundary Poisson law, independently of the 2-dimensional Poisson points in the bulk of the box.

Stationarity was subsequently used by the same authors to give a probabilistic coupling proof that Hammersley's process (and therefore PNG) belongs to the KPZ universality class \cite{Cato_Groe_06}. This seminal work was the starting point for a long sequence of probabilistic treatment of KPZ universality in many related models -- of the numerous references we cite \cite{Bala_Cato_Sepp_06,Bala_Sepp_09,Sepp_12_corr,Emra_Janj_Sepp_23}.

Recently, Drillick and Lin \cite[Theorem 1.8]{Dril_Lin_24} constructed the stationary version of $t$-PNG. Somewhat similarly to Hammersley's process, stationarity happens when the sources and sinks form Poisson processes with respective intensities $\lambda$ and $\frac{1}{\lambda(1-t)}$ for some $\lambda>0$. 

We finally mention yet another interpretation of Hammersley's process due to Aldous and Diaconis \cite{Aldo_Diac_95}. They studied the following continuous-time interacting particle system formulation. Particles are laid down as disjoint points in \(\mathbb R_{\ge0}\). Between any two neighbouring points, and between the leftmost point and the origin, there is a Poisson clock with rate equal the distance of these neighbours. When the clock rings, a uniform point \(U\) is chosen in the interval between the neighbours, and the right neighbour is instantaneously moved to this uniform location. The space-time trajectories of the particles exactly draw the ray ensembles in \(\mathbb R_{\ge0}^2\) with the uniform targets of particle jumps being the 2-dimensional Poisson points of Hammersley's process. The length \(\mrL_{\mathbf0,\mathbf v}\) of the longest chain is obtained as the number of particles in the interval \([(0,v_2),\,\mathbf v]\) between space coordinates 0 and \(v_1\) at time \(v_2\).

One of the properties of Hammersley's and other growth processes is convergence to the stationary picture. This is closely related to \emph{Busemann functions}, considered in the Hammersley process context by Aldous and Diaconis \cite{Aldo_Diac_95} as well as Cator and Pimentel \cite{Cato_Pime_2012_Busemann}. It can be shown that if \(\mathbf v\to\infty\) in a given direction then, 
for any \(\mathbf u\) and \(\widehat{\mathbf u}\), the difference
$
 \mrL_{\mathbf u,\mathbf v}-\mrL_{\widehat{\mathbf u},\mathbf v}
$
converges a.s., and the limit is called the Busemann function \(B(\mathbf u,\,\widehat{\mathbf u})\) of the two points \(\mathbf u\) and \(\widehat{\mathbf u}\). This happens because the chains achieving maximal point count eventually coalesce as they run towards the distant point \(\mathbf v\). Busemann functions have also been investigated in \emph{first passage percolation} (Hoffman \cite{Hoff_2005_coex}), \emph{last passage percolation} (LPP) (Cator and Pimentel \cite{Cato_Pime_2012_Busemann}, Georgiou, Rassoul-Agha, Seppäläinen \cite{Geo_Rass_Sepp_2017_cocy}) and \emph{directed polymer models} (Georgiou, Rassoul-Agha, Seppäläinen, Yilmaz \cite{Geor_Rass_Sepp_Yilm_15}, Janjigian, Rassoul-Agha \cite{Janj_Rass_20}). Matching of particles in Hammersley's process by Aldous and Diaconis \cite{Aldo_Diac_95} and, in a similar spirit, coalescence properties of the geodesics of LPP by Bal\'azs, Busani and Sepp\"al\"ainen \cite{bal_bus_sepp_2021_stabi} were the main tools to show stabilisation. Namely, for this latter, the last passage time differences from the origin to distant points in a given direction jointly converge to those coming from a stationary LPP model. For further discussion on Busemann functions, we refer to the surveys \cite{Auff_Damr_Hans_17, Rass_18, Sepp_18_CGM}. We do not investigate Busemann functions for \(t\)-PNG; this is left for future work.

\subsection{Results and methods}\label{sc:resmeth}

We start with the formal definition of \(t\)-PNG; we will use the Hammersley's process version of Poisson points representation on \(\mathbb R^2_{\ge0}\). Fix \(0\le t<1\) and define the region \(R_d:\,=\{(x,\,y)\in\mathbb R^2_{\ge0}\,:\,x+y\le d\}\), bounded by the simplex at \(L^1\)-distance \(d>0\) from the origin. Place a homogeneous unit intensity Poisson process on \(\mathbb R_{\ge0}^2\), we sometimes refer to these as \emph{nucleation points}. A.s., for small enough \(d\), \(R_d\) will contain no Poisson points, and no lines will intersect this region. As \(d\) grows, at some value the simplex will hit the first of the Poisson points. From there on, start drawing both a horizontal and vertical straight line going right and up from this Poisson point. When, as \(d\) grows further, the simplex meets the subsequent Poisson points, start drawing new horizontal and vertical straight lines going right and up from there as well. Eventually some of these lines will meet. In classical PNG, at this point the meeting lines would annihilate i.e., we would stop drawing these particular lines forward (but we'd keep drawing more of all the other ones). In  contrary, for \(t\)-PNG, we flip a biased coin that is independent of both the Poisson process and of other coins in similar roles. The coin comes up tails with probability \(1-t\), in which case the meeting lines annihilate each other just like in PNG, and we stop drawing them further. The meeting point is then called a \textit{corner point}. However, with probability \(t\) the coin comes up heads, in which case the lines cross each other, forming a \textit{crossing point}. We keep drawing them further to the right and up, respectively, as we cover more and more of \(\mathbb R_{\ge0}^2\) with the region \(R_d\). These lines then behave like those originating from Poisson points: when they meet any other lines, we flip the coin again as described above. This defines \(t\)-PNG with \emph{zero boundaries}; see Figure \ref{fig:zero-boundary-sampling}.

\begin{figure}[ht]
 \begin{center}
  \begin{tikzpicture}[scale=0.5]
  \begin{scope}
   \draw(0,11)--(0,0)--(11,0);
   \node at (1, 5) {$\times$};
   \node at (2.4, 2.5) {$\times$};
   \node at (1.8, 7.9) {$\times$};
   \node at (4, 0.5) {$\times$};
   \node at (5.8, 2) {$\times$};
   \node at (6, 3.8) {$\times$};
   \node at (7, 7) {$\times$};
   \node at (8.5, 4.5) {$\times$};
   \node at (9.5, 1.4) {$\times$};
   \node at (5, 9) {$\times$};
   \node at (10, 10) {$\times$};

   \draw[thick,dotted](0,4.5)--(4.5,0);
   \draw[->](0,4.5)--(.5,5);
   \draw[->](4.5,0)--(5,.5);
   \node at (-.5,4.5) {$R_d$};
   \draw[->](4,.5)--(4,1);
   \draw[->](4,.5)--(4.5,.5);
   \end{scope}
   \begin{scope}[xshift=14cm]
       \draw(0,11)--(0,0)--(11,0);
   \node at (1, 5) {$\times$};
   \node at (2.4, 2.5) {$\times$};
   \node at (1.8, 7.9) {$\times$};
   \node at (4, 0.5) {$\times$};
   \node at (5.8, 2) {$\times$};
   \node at (6, 3.8) {$\times$};
   \node at (7, 7) {$\times$};
   \node at (8.5, 4.5) {$\times$};
   \node at (9.5, 1.4) {$\times$};
   \node at (5, 9) {$\times$};
   \node at (10, 10) {$\times$};

   \draw[thick,dotted](0,6.5)--(6.5,0);
   \draw[->](0,6.5)--(.5,7);
   \draw[->](6.5,0)--(7,.5);
   \node at (-.5,6.5) {$R_d$};
   \draw[very thick](4,.5)--(6,.5);
   \draw[very thick](4,.5)--(4,2.5);
   \draw[->](6,.5)--(6.5,.5);
   \draw[->](4,2.5)--(4,3);
   \draw[very thick](2.4,2.5)--(4,2.5);
   \draw[very thick](2.4,2.5)--(2.4,4.1);
   \draw[->](4,2.5)--(4.5,2.5);
   \draw[->](2.4,4.1)--(2.4,4.6);
   \draw[very thick](1,5)--(1.5,5);
   \draw[very thick](1,5)--(1,5.5);
   \draw[->](1.5,5)--(2,5);
   \draw[->](1,5.5)--(1,6);
   \node[draw,very thick,circle] at (4,2.5) {};
   \end{scope}
   \begin{scope}[yshift=-12cm]
   \draw(0,11)--(0,0)--(11,0);
   \node at (1, 5) {$\times$};
   \node at (2.4, 2.5) {$\times$};
   \node at (1.8, 7.9) {$\times$};
   \node at (4, 0.5) {$\times$};
   \node at (5.8, 2) {$\times$};
   \node at (6, 3.8) {$\times$};
   \node at (7, 7) {$\times$};
   \node at (8.5, 4.5) {$\times$};
   \node at (9.5, 1.4) {$\times$};
   \node at (5, 9) {$\times$};
   \node at (10, 10) {$\times$};

   \draw[thick,dotted](0,7.4)--(7.4,0);
   \draw[->](0,7.4)--(.5,7.9);
   \draw[->](7.4,0)--(7.9,.5);
   \node at (-.5,7.4) {$R_d$};
   \draw[very thick](4,.5)--(6.9,.5);
   \draw[very thick](4,.5)--(4,3.4);
   \draw[->](6.9,.5)--(7.4,.5);
   \draw[->](4,3.4)--(4,3.9);
   \draw[very thick](2.4,2.5)--(4.9,2.5);
   \draw[very thick](2.4,2.5)--(2.4,5);
   \draw[->](4.9,2.5)--(5.4,2.5);
   \draw[very thick](1,5)--(2.4,5);
   \draw[very thick](1,5)--(1,6.4);
   \draw[->](1,6.4)--(1,6.9);
   \node[draw,very thick,circle] at (2.4,5) {};
   \end{scope}
   \begin{scope}[xshift=14cm,yshift=-12cm]
       \draw(0,11)--(0,0)--(11,0);
   \node at (1, 5) {$\times$};
   \node at (2.4, 2.5) {$\times$};
   \node at (1.8, 7.9) {$\times$};
   \node at (4, 0.5) {$\times$};
   \node at (5.8, 2) {$\times$};
   \node at (6, 3.8) {$\times$};
   \node at (7, 7) {$\times$};
   \node at (8.5, 4.5) {$\times$};
   \node at (9.5, 1.4) {$\times$};
   \node at (5, 9) {$\times$};
   \node at (10, 10) {$\times$};

   \node at (1.5, 1.5) {$0$};
   \node at (2.5, 6) {$1$};
   \node at (5, 1.5) {$1$};
   \node at (5, 3.5) {$2$};
   \node at (3, 9) {$2$};
   \node at (10.5, 1.7) {\small$2$};
   \node at (7, 5.5) {$3$};
   \node at (10.5, 3) {$3$};
   \node at (8, 8.5) {$4$};
   \node at (10.5, 4.17) {\small$4$};
   \node at (10.5, 7.5) {$5$};
   \node at (10.5, 10.5) {$6$};

   \draw[very thick](4,.5)--(11,.5);
   \draw[very thick](4,.5)--(4,11);
   \draw[very thick](2.4,2.5)--(5.8,2.5);
   \draw[very thick](2.4,2.5)--(2.4,5);
   \draw[very thick](1,5)--(2.4,5);
   \draw[very thick](1,5)--(1,11);
   \draw[very thick](1.8,7.9)--(6,7.9);
   \draw[very thick](1.8,7.9)--(1.8,11);
   \draw[very thick](6,3.8)--(11,3.8);
   \draw[very thick](6,3.8)--(6,7.9);
   \draw[very thick](10,10)--(11,10);
   \draw[very thick](10,10)--(10,11);
   \draw[very thick](5,9)--(5,11);
   \draw[very thick](5,9)--(7,9);
   \draw[very thick](9.5,1.4)--(11,1.4);
   \draw[very thick](9.5,1.4)--(9.5,11);
   \draw[very thick](5.8,2)--(5.8,2.5);
   \draw[very thick](5.8,2)--(11,2);
   \draw[very thick](7,7)--(7,9);
   \draw[very thick](7,7)--(8.5,7);
   \draw[very thick](8.5,4.5)--(8.5,7);
   \draw[very thick](8.5,4.5)--(11,4.5);
   \end{scope}
  \end{tikzpicture}
  \caption{Sampling process of $t$-PNG with zero boundaries. Circles indicate moments when coin tosses happen. The numbers in the last picture indicate the values of the height function.}\label{fig:zero-boundary-sampling}
 \end{center}
\end{figure}

We now also add \emph{sources} and \emph{sinks} to this definition, to obtain a model with non-zero boundary (or boundaries). A source is simply a vertical line segment that we start drawing from a point \((x,\,0)\) (\(x>0\)), and a sink is a horizontal line segment that we start from \((0,\,y)\) for some \(y>0\). These segments are then continued according to the rules as described above i.e., the sources and sinks act as nucleation points on the boundaries of \(\mathbb R^2_{\ge0}\); see Figure \ref{fig:boundary-example}. We shall see models with either or both sources and sinks, and these will also form the basis of stationary versions of \(t\)-PNG.

\begin{figure}[ht]
 \begin{center}
  \begin{tikzpicture}[scale=0.5]
       \draw(0,11)--(0,0)--(11,0);
   \node at (1, 5) {$\times$};
   \node at (2.4, 2.5) {$\times$};
   \node at (1.8, 7.9) {$\times$};
   \node at (4, 0.5) {$\times$};
   \node at (5.8, 2) {$\times$};
   \node at (6, 3.8) {$\times$};
   \node at (7, 7) {$\times$};
   \node at (8.5, 4.5) {$\times$};
   \node at (9.5, 1.4) {$\times$};
   \node at (5, 9) {$\times$};
   \node at (10, 10) {$\times$};

   \node at (1.5, 0) {$\times$};
   \node at (5, 0) {$\times$};
   \node at (8, 0) {$\times$};
   \node at (0, 4) {$\times$};
   \node at (0, 9.5) {$\times$};

   \node at (.8, 1.5) {$0$};
   \node at (2.5, 1.5) {$1$};
   \node at (2.5, 6) {$2$};
   \node at (4.5, 1.5) {$2$};
   \node at (.5, 10.2) {$2$};
   \node at (6.5, 0.2) {\tiny$2$};
   \node at (5, 3.5) {$3$};
   \node at (3, 8.7) {$3$};
   \node at (1.4, 10.2) {$3$};
   \node at (3, 10.2) {$4$};
   \node at (7, 3) {$4$};
   \node at (10.5, 1.7) {\small$4$};
   \node at (5, 10.2) {$5$};
   \node at (7, 5.5) {$5$};
   \node at (10.5, 3) {$5$};
   \node at (8, 8.5) {$6$};
   \node at (10.5, 4.17) {\small$6$};
   \node at (10.5, 7.5) {$7$};
   \node at (10.5, 10.5) {$8$};

   \draw[very thick](4,.5)--(8,.5);
   \draw[very thick](4,.5)--(4,11);
   \draw[very thick](2.4,2.5)--(5,2.5);
   \draw[very thick](2.4,2.5)--(2.4,5);
   \draw[very thick](1,5)--(2.4,5);
   \draw[very thick](1,5)--(1,11);
   \draw[very thick](1.8,7.9)--(5.8,7.9);
   \draw[very thick](1.8,7.9)--(1.8,11);
   \draw[very thick](6,3.8)--(11,3.8);
   \draw[very thick](6,3.8)--(6,11);
   \draw[very thick](10,10)--(11,10);
   \draw[very thick](10,10)--(10,11);
   \draw[very thick](5,9)--(5,9.5);
   \draw[very thick](5,9)--(7,9);
   \draw[very thick](9.5,1.4)--(11,1.4);
   \draw[very thick](9.5,1.4)--(9.5,11);
   \draw[very thick](5.8,2)--(5.8,7.9);
   \draw[very thick](5.8,2)--(11,2);
   \draw[very thick](7,7)--(7,9);
   \draw[very thick](7,7)--(8.5,7);
   \draw[very thick](8.5,4.5)--(8.5,7);
   \draw[very thick](8.5,4.5)--(11,4.5);

   \draw[very thick](8,0)--(8,.5);
   \draw[very thick](5,0)--(5,2.5);
   \draw[very thick](1.5,0)--(1.5,4);
   \draw[very thick](0,4)--(1.5,4);
   \draw[very thick](0,9.5)--(5,9.5);
  \end{tikzpicture}
  \caption{$t$-PNG with sources and sinks. The numbers indicate the values of the height function.}\label{fig:boundary-example}
 \end{center}
\end{figure}

Whether with zero or non-zero boundaries, the \emph{height} of a point in \(\mathbb R_{\ge0}^2\) is defined to be the smallest number of lines that need to be crossed to reach this point from \((0,\,0)\) avoiding corners and crossing points; see Figure \ref{fig:zero-boundary-sampling}. If the point is on a line, this line is counted towards this point's height.

While the classical case of $t$-PNG model involves a Poisson point process of nucleations on $\bbR^2_{>0}$, it is particularly interesting to consider its augmentations with sources and sinks i.e., additional nucleations on the bottom and left boundaries, where these boundary nucleations form Poisson processes on the boundaries.

For $b,l \ge 0$, denote as $t\text{-PNG}(b, l)$ the $t$-PNG model with additional boundary nucleations such that:
\begin{itemize}
    \item there is a Poisson process of boundary nucleations on the bottom boundary of rate $b$,
    \item there is a Poisson process of boundary nucleations on the left boundary of rate $l$,
    \item these two Poisson processes and the Poisson process of nucleations in the bulk are jointly independent. 
\end{itemize}

Thus, for a random $t$-PNG configuration $\phi$ distributed as described above, we write $\phi \sim t\text{-PNG}(b, l)$. As shown in \cite[Theorem 1.8]{Dril_Lin_24}, a $t$-PNG model is stationary if and only if it is of form $t\text{-PNG}(\lambda, \frac{1}{\lambda(1-t)})$ for some $\lambda > 0$.

In this paper we establish the local convergence in zero-boundary $t$-PNG as conjectured by Drillick and Lin \cite[Remark 1.4]{Dril_Lin_24}. That is, we show that the $t$-PNG model with zero-boundary converges in the local limit to a stationary $t$-PNG model of parameter $\lambda>0$ that depends explicitly on the direction of convergence.

We now formulate this precisely. Denote the union of the \(t\)-PNG lines with zero boundaries by \(\theta\). Fix a real parameter \(\lambda>0\), and consider a sequence $(\bfv_k = (x_k, y_k))_{k\in\bbZ_{>0}}$ of points in $\bbR^2_{>0}$ such that $|\bfv_k| \to \infty$ and $\frac{\bfv_k}{|\bfv_k|} \to \bfv = (x,y)$ with $\frac{y}{x} = \lambda^2(1-t)$. Let $(\mathcal{I}^i)_{i \in [M]}$ be a family of disjoint non-empty intervals of $\bbR_{>0} \times\{0\}$ for some $M \in \bbZ_{>0}$. We write $m_i$ for the length of the interval $\mathcal{I}^i$. Similarly, let $(\mathcal{J}^j)_{j \in [N]}$ be a family of disjoint non-empty intervals of $\{0\}\times\bbR_{>0}$ for some $N \in \bbZ_{>0}$, and let $n_j$ be the length of $\mathcal{J}^j$. Write $\mathcal{I}^i_k$ and $\mathcal{J}^j_k$, respectively, for the intervals $\mathcal{I}^i + \bfv_k$ and $\mathcal{J}^j + \bfv_k$ shifted by $\bfv_k$. Here and further, when talking about the \emph{intersection} between a $t$-PNG model and any deterministic horizontal or vertical interval, we only consider those lines of that $t$-PNG model which are perpendicular to this interval. Note that the probability that the interval overlaps with a similarly oriented line segment of the \(t\)-PNG is zero i.e., this a.s.\ coincides with intersections as sets. In particular, the \(t\)-PNG lines will a.s.\ intersect the aforementioned intervals $\mathcal{I}^i_k$ and $\mathcal{J}^j_k$ in finitely many points.

\begin{thm} 
\label{th:LocConv}
The intersections of the lines of a zero-boundary $t$-PNG process $\theta$ with horizontal and vertical intervals of constant size, taken in a limiting direction with the slope $\frac{y}{x}$, converge in distribution to Poisson processes of rate $\frac{1}{\sqrt{1-t}}\sqrt{\frac{y}{x}}$ for horizontal intervals and $\frac{1}{\sqrt{1-t}}\sqrt{\frac{x}{y}}$ for vertical intervals. Formally, under previous notation, for a sequence $(\bfv_k = (x_k, y_k))_{k\in\bbZ_{>0}}$ of points in $\bbR^2_{>0}$ satisfying $|\bfv_k| \to \infty$ and $\frac{\bfv_k}{|\bfv_k|} \to \bfv = (x,y)$, for arbitrary interval families $(\mathcal{I}^i)_{i \in [M]}$ and $(\mathcal{J}^j)_{j \in [N]}$,
\begin{multline*}
    \left( \left|\theta \cap \mathcal{I}^i_k\right| \right)_{i \in [M]} \otimes \left( \left|\theta \cap \mathcal{J}^j_k\right| \right)_{j \in [N]}\\
    \xrightarrow{d} \left(\bigotimes\limits_{i \in [M]} \Pois\left[\frac{m_i}{\sqrt{1-t}}\sqrt{\frac{y}{x}}\right]\right) \otimes \left(\bigotimes\limits_{j \in [N]} \Pois\left[\frac{n_j}{\sqrt{1-t}}\sqrt{\frac{x}{y}}\right]\right)
\end{multline*}
as $k \to \infty$.
\end{thm}

For $t = 0$, the conjecture specializes to a result of Aldous and Diaconis \cite{Aldo_Diac_95}. 

Our proof makes use of a triple coupling $(\omega, \eta, \alpha)$ of $t$-PNG processes parametrized by a positive number $\varepsilon$.
\begin{itemize}
    \item $\omega$ is a $t$-PNG with Poisson nucleations of rate $\lambda+\varepsilon$ on the bottom boundary and a zero left boundary condition.
    \item $\eta$ is a stationary $t$-PNG of parameter $\lambda+\varepsilon$.
    \item $\alpha$ is a stationary $t$-PNG of parameter $\lambda$.
\end{itemize}
One can establish an increasing coupling between them, i.e.\ so that $\omega \ge \eta \ge \alpha$. Here, for any two $t$-PNG processes, say $\phi$ and $\psi$, $\phi \le \psi$ if the union of all vertical segments of $\phi$ is a subset of the union of all vertical segments of $\psi$, and the union of all horizontal segments of $\psi$ is a subset of the union of all horizontal segments of $\phi$. This allows us to consider second class particle processes between $\omega$ and $\eta$ (denoted as $\frac{\omega}{\eta}$) and between $\omega$ and $\alpha$ (denoted as $\frac{\omega}{\alpha}$). These processes are also ordered: $\frac{\omega}{\eta} \le \frac{\omega}{\alpha}$. (See Section \ref{S:2CP} for precise definitions). 

The setup is somewhat reminiscent to second and \emph{third} class particles in ASEP, given that the lines of $t$-PNG can be interpreted as the space-time trajectories of first-class particles, the horizontal coordinate being space and the vertical being time. In that language, \(\frac\omega\alpha\) only would be second class particles, while \(\frac\omega\alpha\) sharing a site with \(\frac\omega\eta\) would correspond to a third class particle. When these become neighbours in ASEP, and the \(\eta\) process moves, a second class particle pushes the third class one behind. Notice though that here we do not have a discrete underlying lattice to \(t\)-PNG. This picture nevertheless becomes handy and allows us to introduce couplings between second class particles in the spirit of Ferrari, Kipnis and Saada \cite{fks_91}, see also \cite{Bala_Sepp_09} and, in the six-vertex model context, Drillick and Haunschmid-Sibitz \cite{dril_h-sibitz_stoch_6v_speed_p} where similar coupling ideas were applied. Reversible blocking measures also play a role in these works, just as in our paper.

Since $\eta$ is stationary, its increments are distributed as Poisson processes. In a given direction, the convergence of increments of $\omega$ to Poisson processes in distribution is obtained by bounding the impact of the second class particles $\frac{\omega}{\eta}$, with a suitably chosen parameter $\lambda$. For this, we notice that each of these second-class particles starts from the left boundary and moves together with some particle of $\frac{\omega}{\alpha}$, jumping to the next or previous $\frac{\omega}{\alpha}$-particle according to a process equivalent to a discrete-time ASEP. We then conclude the tightness of the right tail of the rightmost particle's position in this ASEP by coupling the model with another ASEP with a stationary blocking measure. This implies that the probability of the increment of $\omega$ in the direction of interest being impacted by second-class particles goes to $0$. Once the limit distribution of the increments of $\omega$ is obtained, we let $\varepsilon \to 0$, thus making $\eta$ and, by the above second class particle argument, $\omega$ behave like the stationary model $\alpha$ under the direction of interest. Then, the one-sided model \(\omega\) can be compared to a zero boundary model which finishes the proof of the theorem.

\subsection{Organization of the paper}

The rest of this article is organized as follows. Section \ref{S:2CP} covers the increasing coupling of two $t$-PNG models and the associated second-class particles that represent the discrepancy between the two models. Section \ref{S:Height} computes the law of large numbers limit of the height functions of the stationary $t$-PNG as well as the $t$-PNG with one-sided boundary. Section \ref{S:LLN_2CP} states and proves Theorem \ref{th:2nd-class-lln}, which computes the asymptotic direction of a tagged second-class particle between  the stationary $t$-PNG and the $t$-PNG with one-sided boundary. This result is one of the main ingredients for our proof of local convergence in Theorem \ref{th:LocConv}. To establish this result, Section \ref{S:PfLocConv} also develops the previously mentioned $(\omega, \eta, \alpha)$-coupling and records a key geometric tail bound on the labels of the $\frac{\omega}{\alpha}$-particles that share their trajectory with the first $\frac{\omega}{\eta}$-particle. 

\subsection{Acknowledgements}

 We thank an anonymous referee for the careful reading of the first version of this manuscript and helping its evolution with insightful comments.

 R.B.\ and A.B.\ were supported by the School of Mathematics of the University of Bristol and the Heilbronn Institute for Mathematical Research. M.B.\ was partially, while E.E.\ was fully supported by the EPSRC grant EP/W032112/1. This study did not involve any underlying data.

\section{Second class particles}
\label{S:2CP}

It is often useful to be able to couple two or more $t$-PNG processes with different boundary data, maintaining a certain order between them. This is where the notion of second class particles between two models is of great help. This section is dedicated to the definition and description of second class particles.

Throughout the article, we will couple in a way that the bulk nucleation points are the same for every \(t\)-PNG process we consider.

We now introduce a partial order on $t$-PNG processes with arbitrary boundary data. Recall that by intersections of a $t$-PNG process and a horizontal or vertical interval we only mean the perpendicular intersections. For $t$-PNG processes $\phi$ and $\psi$, we write $\phi \le \psi$ if, for any horizontal interval $\mathcal{I}$ and vertical interval $\mathcal{J}$, 
\begin{align*}
(\phi \cap \mathcal{I}) \subseteq (\psi \cap \mathcal{I}) \text{ and } (\phi \cap \mathcal{J}) \supseteq (\psi \cap \mathcal{J}).
\end{align*}

Write $\phi_\uparrow$ for the set of \textit{sources} of $\phi$, that is, nucleations on the horizontal boundary that emit only upward rays. Similarly, write $\phi_\rightarrow$ for the set of \textit{sinks} of $\phi$, that is, nucleations on the vertical boundary that emit only rightward rays. Then $\phi \le \psi$ implies that $\phi_\uparrow \subseteq \psi_\uparrow$ and $\phi_\rightarrow \supseteq \psi_\rightarrow$. $\phi \le \psi$ also means that all vertical segments of $\phi$ are covered by vertical segments of $\psi$, and all horizontal segments of $\psi$ are covered by horizontal segments of $\phi$. 

An \textit{increasing coupling} between $t$-PNG processes $\phi$ and $\psi$ is a coupling that maintains the ordering $\phi \le \psi$. We do not emphasise the particle interpretation of our models here, but in that picture this coupling would correspond to \emph{attractivity} of the models. The next lemma establishes the existence of such a coupling.

\begin{lem}
\label{t-PNG coupling}
    Let $\phi$ and $\psi$ be $t$-PNG models with fixed bulk nucleations, sources $\phi_\uparrow$, $\psi_\uparrow$ and sinks $\phi_\rightarrow$, $\psi_\rightarrow$, respectively. Suppose that $\phi_\uparrow \subseteq \psi_\uparrow$ and $\phi_\rightarrow \supseteq \psi_\rightarrow$. Then there exists an increasing coupling of $\phi$ and $\psi$, i.e.\ such that $\phi \le \psi$. 
\end{lem}
\begin{proof}
    A similar result was proven by Drillick and Lin \cite[Lemma 4.3 (i)]{Dril_Lin_24}. Following their notation, we colour sources $\phi_\uparrow$ and sinks $\phi_\rightarrow$ with colour 1 and colour sources $\psi_\uparrow \setminus \phi_\uparrow$ and sinks $\phi_\rightarrow \setminus \psi_\rightarrow$ with colour 2; note that sinks $\phi_\rightarrow \setminus \psi_\rightarrow$ will be coloured twice. Then, we sample $\phi$ and $\psi$ according to these colours by therein defined rules. The following invariant property is propagated from the boundary: all horizontal segments of colour 2 are also coloured with colour 1, and all vertical segments of colour 2 are not coloured with colour 1; see \cite[Appendix A]{Dril_Lin_24}. This means
    \begin{itemize}
        \item colour 2 never erases vertical segments of colour 1, which implies $(\phi \cap \mathcal{I}) \subseteq (\psi \cap \mathcal{I})$ for every horizontal segment $\mathcal{I}$;
        \item colour 2 never creates new horizontal segments; it only erases those of colour 1. This implies $(\phi \cap \mathcal{J}) \supseteq (\psi \cap \mathcal{J})$ for every vertical segment $\mathcal{J}$.
    \end{itemize}
    The needed coupling is found.
\end{proof}

\begin{cor}\label{cor:poi-coupling}
    Let $\phi \sim t\text{-\rm{PNG}}(b_\phi, l_\phi)$ and $\psi \sim t\text{-\rm{PNG}}(b_\psi, l_\psi)$ with $b_\phi \le b_\psi$ and $l_\phi \ge l_\psi$. Then there exists an increasing coupling of $\phi$ and $\psi$ such that $\phi \le \psi$.
\end{cor}
\begin{proof}
    We sample $\phi_\uparrow$, the Poisson process of bottom boundary nucleations of $\phi$, as a thinning of $\psi_\uparrow$. Similarly, $\psi_\rightarrow$ shall be sampled as a thinning of $\phi_\rightarrow$. This guarantees $\phi_\uparrow \subseteq \psi_\uparrow$ and $\phi_\rightarrow \supseteq \psi_\rightarrow$, hence the required coupling exists by Lemma \ref{t-PNG coupling}.
\end{proof}

When comparing $t$-PNG processes with common bulk nucleations, it is convenient to work with the concept of second class particles, which is already extensively studied for classical $0$-PNG; see Cator and Groeneboom \cite{Cato_Groe_05} and \cite{Cato_Groe_06}. Informally speaking, second class particles of a pair of increasingly coupled processes characterize the difference between these processes. As noted in \cite{Dril_Lin_24}, the lines of colour 2 in the proof of Lemma \ref{t-PNG coupling} can be viewed as paths of second class particles between $\psi$ and $\phi$. We will now describe the behaviour of these particles.

Suppose $\phi$, $\psi$ are $t$-PNG processes coupled so that $\phi \le \psi$. We define the \textit{second class line ensemble (SCLE)} $\frac{\psi}{\phi}$ on the $t$-PNG diagram as the closure of the symmetric difference of all lines of $\phi$ and $\psi$, so that for any horizontal interval $\mathcal{I}$ and vertical interval $\mathcal{J}$
\begin{align*}
\frac{\psi}{\phi} \cap \mathcal{I} = (\psi \cap \mathcal{I}) \setminus (\phi \cap \mathcal{I}) \text{ and } \frac{\psi}{\phi} \cap \mathcal{J} = (\phi \cap \mathcal{J}) \setminus (\psi \cap \mathcal{J}).    
\end{align*}
The closure makes the definition more convenient to work with, ensuring that the SCLE does not miss the points where the lines of $\phi$ and $\psi$ intersect perpendicularly; namely, the central point in each of the diagrams in Figure \ref{fig:scp-turns}. Similarly to $t$-PNG, we write $\left(\frac{\psi}{\phi}\right)_\uparrow$ and $\left(\frac{\psi}{\phi}\right)_\rightarrow$ for the sets of points where the lines of SCLE intersect the horizontal and vertical boundaries, respectively; thus, $\left(\frac{\psi}{\phi}\right)_\uparrow = \psi_\uparrow \setminus \phi_\uparrow$ and $\left(\frac{\psi}{\phi}\right)_\rightarrow = \phi_\rightarrow \setminus \psi_\rightarrow$.

As indicated by Figure \ref{fig:scp-turns}, the SCLE can be decomposed into up-right paths which consist of alternating horizontal and vertical segments, start at the boundary, and continue indefinitely. The only ambiguity is in case (a); in this case, the bottom segment continues into the right segment, and the left segment continues into the top segment. This ensures that no pair of such paths cross each other, though they can intersect at a countable set of points where they turn (see Figure \ref{fig:scp-turns}.a).

The \textit{second class particle system} between $t$-PNG models $\psi$ and $\phi$ is an ensemble of particles moving on $\bbR^2_{\ge 0}$ whose trajectories are given by the up-right paths of the second class line ensemble $\frac{\psi}{\phi}$. We refer to these particles as \textit{second class particles} or \textit{$\frac{\psi}{\phi}$-particles}. Every $\frac{\psi}{\phi}$-particle moves along $\psi$-lines on all vertical segments of its path and along $\phi$-lines on all horizontal segments of its path.

The rule indicated in Figure \ref{fig:scp-turns}.a guarantees that when two path segments of $\frac{\psi}{\phi}$-particles, one vertical and one horizontal, meet, each of them changes direction, and thus the up-down and left-right ordering of second class particles on the $t$-PNG diagram is preserved. 

\begin{figure}[ht]
 \begin{center}
  \begin{tikzpicture}[scale=0.55]
   \begin{scope}[xshift=0cm]
    \draw(0,0)--(4,0)node[midway,below=8pt]{(a)}--(4,4)--(0,4)--cycle;
    \draw[red,very thick](4,2)--(0,2);
    \draw[blue,very thick](2,0)--(2,4);
   \end{scope}
   \begin{scope}[xshift=5cm]
    \draw(0,0)--(4,0)node[midway,below=8pt]{(b)}--(4,4)--(0,4)--cycle;
    \draw[red,very thick](1.9,0)--(1.9,1.9)--(0,1.9);
    \draw[blue,very thick](2.1,0)--(2.1,4);
   \end{scope}
   \begin{scope}[xshift=10cm]
    \draw(0,0)--(4,0)node[midway,below=8pt]{(c)}--(4,4)--(0,4)--cycle;
    \draw[red,very thick](1.9,0)--(1.9,4);
    \draw[red,very thick](0,1.9)--(4,1.9);
    \draw[blue,very thick](2.1,0)--(2.1,4);
   \end{scope}
   \begin{scope}[xshift=15cm]
    \draw(0,0)--(4,0)node[midway,below=8pt]{(d)}--(4,4)--(0,4)--cycle;
    \draw[red,very thick](4,2.1)--(0,2.1);
    \draw[blue,very thick](1.9,0)--(1.9,1.9)--(0,1.9);
   \end{scope}
   \begin{scope}[xshift=20cm]
    \draw(0,0)--(4,0)node[midway,below=8pt]{(e)}--(4,4)--(0,4)--cycle;
    \draw[red,very thick](4,2.1)--(0,2.1);
    \draw[blue,very thick](0,1.9)--(4,1.9);
    \draw[blue,very thick](1.9,0)--(1.9,4);
   \end{scope}
   \begin{scope}[xshift=0cm, yshift=-6cm]
    \draw(0,0)--(4,0)node[midway,below=8pt]{(a)}--(4,4)--(0,4)--cycle;
    \draw[black,very thick](2,0)--(2,4);
    \draw[black,very thick](0,2)--(4,2);
    \draw[->,black,thick](1.2,2.2)--(1.8,2.2)--(1.8,2.8);
    \draw[->,black,thick](2.2,1.2)--(2.2,1.8)--(2.8,1.8);
   \end{scope}
   \begin{scope}[xshift=5cm, yshift=-6cm]
    \draw(0,0)--(4,0)node[midway,below=8pt]{(b)}--(4,4)--(0,4)--cycle;
    \draw[black,very thick](0,2)--(2,2)--(2,4);
   \end{scope}
   \begin{scope}[xshift=10cm, yshift=-6cm]
    \draw(0,0)--(4,0)node[midway,below=8pt]{(c)}--(4,4)--(0,4)--cycle;
    \draw[black,very thick](0,2)--(4,2);
   \end{scope}
   \begin{scope}[xshift=15cm, yshift=-6cm]
    \draw(0,0)--(4,0)node[midway,below=8pt]{(d)}--(4,4)--(0,4)--cycle;
    \draw[black,very thick](2,0)--(2,2)--(4,2);
   \end{scope}
   \begin{scope}[xshift=20cm, yshift=-6cm]
    \draw(0,0)--(4,0)node[midway,below=8pt]{(e)}--(4,4)--(0,4)--cycle;
    \draw[black,very thick](2,0)--(2,4);
   \end{scope}
  \end{tikzpicture}
  \caption{Possible configurations of two coupled $t$-PNG models $\phi \le \psi$ (above) and the corresponding second class particle configurations (below). Here and further, $\phi$ is in red, $\psi$ is in blue, and $\frac{\psi}{\phi}$ is in black.}\label{fig:scp-turns}
 \end{center}
\end{figure}

For fixed $\phi$ and fixed set $N$ of boundary nucleations of $\psi$, consistent with $\phi \le \psi$, the process of second class particles can be used for sampling $\psi$ under the conditions that $\phi \le \psi$ and $\psi_\uparrow \cup \psi_\rightarrow = N$. Apart from collisions with other $\frac{\psi}{\phi}$-particles (Figure \ref{fig:scp-turns}.a), the path of each $\frac{\psi}{\phi}$-particle will be sampled as follows: 
\begin{itemize}
    \item A horizontally moving $\frac{\psi}{\phi}$-particle always turns up when it meets a corner point of $\phi$, that is, the end of the horizontal $\phi$-segment it moves along (Figure \ref{fig:scp-turns}.b). However, if it meets a crossing point of \(\phi\), it will continue straight across it (Figure \ref{fig:scp-turns}.c).
    \item A vertically moving $\frac{\psi}{\phi}$-particle turns right with probability $1-t$ when it meets a horizontal $\phi$-line, and starts to follow this $\phi$-line (Figure \ref{fig:scp-turns}.d). With probability $t$, it does not turn and keeps moving vertically (Figure \ref{fig:scp-turns}.e).
\end{itemize}

Label second class particles with integers, ordering them by their starting points on the boundary, with decreasing non-positive integers from bottom to top along the left boundary and increasing positive integers from left to right along the bottom boundary; thus, the lowest particle to start from the left boundary is indexed by $0$. As noted above, this ordering is invariant, that is, for every $i < j$ the entire path of particle $j$ lies below and to the right of the path of particle $i$, with the exception of the meeting points shown in Figure \ref{fig:scp-turns}.a in the case $j=i+1$.

See Figure \ref{fig:config} for a possible configuration of $\phi$, $\psi$, and $\frac{\psi}{\phi}$, with the labelling of $\frac{\psi}{\phi}$.

\begin{figure}[ht]
 \begin{center}
  \begin{tikzpicture}[scale=0.5]
  \begin{scope}
   \draw(0,0)--(12,0)--(12,8)--(0,8)--cycle;
   \node[red] at (3, 0) {$\times$};
   \node[red] at (5, 0) {$\times$};
   \node[red] at (8, 0) {$\times$};
   \node[red] at (11.5, 0) {$\times$};
   
   \node[red] at (0, 1) {$\times$};
   \node[red] at (0, 3) {$\times$};
   \node[red] at (0, 4) {$\times$};
   \node[red] at (0, 6) {$\times$};
   \node[red] at (0, 7.5) {$\times$};

   \node at (2, 5) {$\times$};
   \node at (2.4, 2.5) {$\times$};
   \node at (2.8, 6.9) {$\times$};
   \node at (4, 0.5) {$\times$};
   \node at (5.8, 2) {$\times$};
   \node at (6, 3.8) {$\times$};
   \node at (7, 7) {$\times$};
   \node at (8.5, 4.5) {$\times$};
   \node at (9.5, 1.4) {$\times$};
   \node at (11, 5.5) {$\times$};

   \draw[red,very thick](12,0.5)--(4,0.5)--(4,5)--(2,5)--(2,6)--(0,6);
   \draw[red,very thick](12,2)--(5.8,2)--(5.8,2.5)--(2.4,2.5)--(2.4,4)--(0,4);
   \draw[red,very thick](8,0)--(8,3.8)--(6,3.8)--(6,6.9)--(2.8,6.9)--(2.8,7.5)--(0,7.5);
   \draw[red,very thick](3,0)--(3,1)--(0,1);
   \draw[red,very thick](5,0)--(5,3)--(0,3);
   \draw[red,very thick](11.5,0)--(11.5,1.4)--(9.5,1.4)--(9.5,4.5)--(8.5,4.5)--(8.5,7)--(7,7)--(7,8);
   \draw[red,very thick](12,5.5)--(11,5.5)--(11,8);
   \end{scope}
   \begin{scope}[xshift=14cm]
       \draw(0,0)--(12,0)--(12,8)--(0,8)--cycle;
   \node[blue] at (1.5, -0.3) {$\times$};
   \node[blue] at (3, -0.3) {$\times$};
   \node[red] at (3, 0) {$\times$};
   \node[blue] at (5, -0.3) {$\times$};
   \node[red] at (5, 0) {$\times$};
   \node[blue] at (6.3, -0.3) {$\times$};
   \node[blue] at (8, -0.3) {$\times$};
   \node[red] at (8, 0) {$\times$};
   \node[blue] at (10, -0.3) {$\times$};
   \node[blue] at (11.5, -0.3) {$\times$};
   \node[red] at (11.5, 0) {$\times$};
   
   \node[red] at (0, 1) {$\times$};
   \node[blue] at (-0.3, 1) {$\times$};
   \node[red] at (0, 3) {$\times$};
   \node[blue] at (-0.3, 3) {$\times$};
   \node[red] at (0, 4) {$\times$};
   \node[red] at (0, 6) {$\times$};
   \node[blue] at (-0.3, 6) {$\times$};
   \node[red] at (0, 7.5) {$\times$};

   \node at (2, 5) {$\times$};
   \node at (2.4, 2.5) {$\times$};
   \node at (2.8, 6.9) {$\times$};
   \node at (4, 0.5) {$\times$};
   \node at (5.8, 2) {$\times$};
   \node at (6, 3.8) {$\times$};
   \node at (7, 7) {$\times$};
   \node at (8.5, 4.5) {$\times$};
   \node at (9.5, 1.4) {$\times$};
   \node at (11, 5.5) {$\times$};

   \draw[red,very thick](12,0.5)--(4,0.5)--(4,5)--(2,5)--(2,6)--(0,6);
   \draw[red,very thick](12,2)--(5.8,2)--(5.8,2.5)--(2.4,2.5)--(2.4,4)--(0,4);
   \draw[red,very thick](8,0)--(8,3.8)--(6,3.8)--(6,6.9)--(2.8,6.9)--(2.8,7.5)--(0,7.5);
   \draw[red,very thick](3,0)--(3,1)--(0,1);
   \draw[red,very thick](5,0)--(5,3)--(0,3);
   \draw[red,very thick](11.5,0)--(11.5,1.4)--(9.5,1.4)--(9.5,4.5)--(8.5,4.5)--(8.5,7)--(7,7)--(7,8);
   \draw[red,very thick](12,5.5)--(11,5.5)--(11,8);

   \draw[black,very thick](0,7.4)--(1.9,7.4)--(1.9,8);
   \draw[black,very thick](0,3.9)--(1.4,3.9)--(1.4,5.9)--(2.1,5.9)--(2.1,7.4)--(2.3,7.4)--(2.3,8);
   \draw[black,very thick](1.6,0)--(1.6,3.9)--(2.5,3.9)--(2.5,7.4)--(2.9,7.4)--(2.9,8);
   \draw[black,very thick](6.4,0)--(6.4,0.4)--(9.9,0.4)--(9.9,8);
   \draw[black,very thick](10.1,0)--(10.1,0.4)--(12,0.4);
   \node[black] at (-0.5, 7.4) {\small-1};
   \node[black] at (-0.5, 3.9) {\small0};
   \node[black] at (1.6, -0.8) {\small1};
   \node[black] at (6.4, -0.8) {\small2};
   \node[black] at (10.1, -0.8) {\small3};
   \end{scope}
   \begin{scope}[yshift=-10cm]
       \draw(0,0)--(12,0)--(12,8)--(0,8)--cycle;
   \node[blue] at (1.5, 0) {$\times$};
   \node[blue] at (3, 0) {$\times$};
   \node[blue] at (5, 0) {$\times$};
   \node[blue] at (6.3, 0) {$\times$};
   \node[blue] at (8, 0) {$\times$};
   \node[blue] at (10, 0) {$\times$};
   \node[blue] at (11.5, 0) {$\times$};
   
   \node[blue] at (0, 1) {$\times$};
   \node[blue] at (0, 3) {$\times$};
   \node[blue] at (0, 6) {$\times$};
   
   \node at (2, 5) {$\times$};
   \node at (2.4, 2.5) {$\times$};
   \node at (2.8, 6.9) {$\times$};
   \node at (4, 0.5) {$\times$};
   \node at (5.8, 2) {$\times$};
   \node at (6, 3.8) {$\times$};
   \node at (7, 7) {$\times$};
   \node at (8.5, 4.5) {$\times$};
   \node at (9.5, 1.4) {$\times$};
   \node at (11, 5.5) {$\times$};

   \draw[blue,very thick](1.5,0)--(1.5,6)--(0,6);
   \draw[blue,very thick](6.3,0)--(6.3,0.5)--(4,0.5)--(4,5)--(2,5)--(2,8);
   \draw[blue,very thick](12,2)--(5.8,2)--(5.8,2.5)--(2.4,2.5)--(2.4,8);
   \draw[blue,very thick](8,0)--(8,3.8)--(6,3.8)--(6,6.9)--(2.8,6.9)--(2.8,8);
   \draw[blue,very thick](3,0)--(3,1)--(0,1);
   \draw[blue,very thick](5,0)--(5,3)--(0,3);
   \draw[blue,very thick](11.5,0)--(11.5,1.4)--(9.5,1.4)--(9.5,4.5)--(8.5,4.5)--(8.5,7)--(7,7)--(7,8);
   \draw[blue,very thick](12,5.5)--(11,5.5)--(11,8);
   \draw[blue,very thick](10,0)--(10,8);
   \end{scope}
   \begin{scope}[xshift=14cm,yshift=-10cm]
       \draw(0,0)--(12,0)--(12,8)--(0,8)--cycle;
   \node[blue] at (1.5, 0) {$\times$};
   \node[blue] at (3, 0) {$\times$};
   \node[red] at (3, -0.3) {$\times$};
   \node[blue] at (5, 0) {$\times$};
   \node[red] at (5, -0.3) {$\times$};
   \node[blue] at (6.3, 0) {$\times$};
   \node[blue] at (8, 0) {$\times$};
   \node[red] at (8, -0.3) {$\times$};
   \node[blue] at (10, 0) {$\times$};
   \node[blue] at (11.5, 0) {$\times$};
   \node[red] at (11.5, -0.3) {$\times$};
   
   \node[red] at (-0.3, 1) {$\times$};
   \node[blue] at (0, 1) {$\times$};
   \node[red] at (-0.3, 3) {$\times$};
   \node[blue] at (0, 3) {$\times$};
   \node[red] at (-0.3, 4) {$\times$};
   \node[red] at (-0.3, 6) {$\times$};
   \node[blue] at (0, 6) {$\times$};
   \node[red] at (-0.3, 7.5) {$\times$};
   
   \node at (2, 5) {$\times$};
   \node at (2.4, 2.5) {$\times$};
   \node at (2.8, 6.9) {$\times$};
   \node at (4, 0.5) {$\times$};
   \node at (5.8, 2) {$\times$};
   \node at (6, 3.8) {$\times$};
   \node at (7, 7) {$\times$};
   \node at (8.5, 4.5) {$\times$};
   \node at (9.5, 1.4) {$\times$};
   \node at (11, 5.5) {$\times$};

   \draw[blue,very thick](1.5,0)--(1.5,6)--(0,6);
   \draw[blue,very thick](6.3,0)--(6.3,0.5)--(4,0.5)--(4,5)--(2,5)--(2,8);
   \draw[blue,very thick](12,2)--(5.8,2)--(5.8,2.5)--(2.4,2.5)--(2.4,8);
   \draw[blue,very thick](8,0)--(8,3.8)--(6,3.8)--(6,6.9)--(2.8,6.9)--(2.8,8);
   \draw[blue,very thick](3,0)--(3,1)--(0,1);
   \draw[blue,very thick](5,0)--(5,3)--(0,3);
   \draw[blue,very thick](11.5,0)--(11.5,1.4)--(9.5,1.4)--(9.5,4.5)--(8.5,4.5)--(8.5,7)--(7,7)--(7,8);
   \draw[blue,very thick](12,5.5)--(11,5.5)--(11,8);
   \draw[blue,very thick](10,0)--(10,8);

   \draw[black,very thick](0,7.4)--(1.9,7.4)--(1.9,8);
   \draw[black,very thick](0,3.9)--(1.4,3.9)--(1.4,5.9)--(2.1,5.9)--(2.1,7.4)--(2.3,7.4)--(2.3,8);
   \draw[black,very thick](1.6,0)--(1.6,3.9)--(2.5,3.9)--(2.5,7.4)--(2.9,7.4)--(2.9,8);
   \draw[black,very thick](6.4,0)--(6.4,0.4)--(9.9,0.4)--(9.9,8);
   \draw[black,very thick](10.1,0)--(10.1,0.4)--(12,0.4);
   \node[black] at (-0.8, 7.4) {\small-1};
   \node[black] at (-0.8, 3.9) {\small0};
   \node[black] at (1.6, -0.5) {\small1};
   \node[black] at (6.4, -0.5) {\small2};
   \node[black] at (10.1, -0.5) {\small3};
   \end{scope}
  \end{tikzpicture}
 \end{center}
\caption{A possible configuration on a rectangle in the bottom left corner of $\bbR^2_{\ge 0}$, with the labelling of second class particles. Here $\phi \le \psi$, $\phi$, including its boundary nucleation points, is in red, $\psi$ and its boundary nucleation points are in blue, and $\frac{\psi}{\phi}$ is in black.}\label{fig:config}
\end{figure}

For convenience of notation, it is useful to time-parametrise the paths of second class particles in a way that, if two particles pass through the same point, they do it at the same time moment. If $Q$ is a second class particle, we write $Q(\sigma) : \bbR_{\ge 0} \mapsto \bbR^2_{\ge 0}$ for an injective continuous parameterization of the path of $Q$ such that $Q(0)$ is the point where $Q$ enters the system (in our case, either on the bottom or on the left boundary), and for every point $\bfv$ where the paths of two particles $Q$ and $R$ meet, there is a unique $\sigma_\bfv$ for which $Q(\sigma_\bfv) = R(\sigma_\bfv) = \bfv$. An example of a suitable parameterization would be the function that maps $\sigma$ to the unique point on the path of $Q$ with the product of coordinates equal to $\sigma$. It is well-defined for each up-right path with a unique intersection point with the boundary. This notation will be extensively used in the final section.

\section{Height function}
\label{S:Height}

This section states some useful facts about the \textit{height function} of a $t$-PNG process, previously defined in \cite{Dril_Lin_24}. We denote the height of a $t$-PNG process $\zeta$ at a point $\bfv = (x,y)$ as $\mrL_\zeta(\bfv)$ (or $\mrL_\zeta(x, y)$). We recall that $\mrL_\zeta(\bfv)$ can be interpreted as the number of $\zeta$-lines that cross the closed straight line segment between the origin and $\bfv$. Note that the word ``closed" implies right-continuity of the height function along every up-right path.

Define the \textit{mean function} by $\mrM_\bfv(\lambda) \coloneq \bbE[\mrL_\zeta(\bfv)]$ where $\zeta \sim t\text{-PNG}(\lambda, \frac{1}{\lambda(1-t)})$ is a stationary $t$-PNG process with parameter $\lambda$. This function can be expressed as
\begin{align*}
    \mrM_\bfv(\lambda) = x\lambda + \frac{y}{\lambda(1-t)}
\end{align*}
as noted in the proof of \cite[Lemma 4.4]{Dril_Lin_24}. We also introduce $\lambda_\bfv = \sqrt{\dfrac{y}{x(1-t)}}$, which is the unique minimizer of $\mrM_\bfv$ on $\bbR_{\ge 0}$, and the \textit{shape function}
\begin{align*}
    \gamma_\bfv = \min_{\lambda \in \bbR_{\ge 0}} \mrM_\bfv(\lambda) = \mrM_\bfv(\lambda_\bfv) = 2\sqrt{\frac{xy}{1-t}}.
\end{align*}

The following result describes the limiting behaviour of the height function considered far from the origin in some fixed direction in a stationary process. In the statement, the notation $|\cdot|$ refers to the Euclidean norm of a vector. 

\begin{prop}[Strong law of large numbers for height in stationarity]\label{th:height-lln}\ \\
    Let $\zeta \sim t\text{-\rm{PNG}}(\lambda, \frac{1}{\lambda(1-t)})$. Fix $\bfv \in \bbR_{\ge 0}^2$ and let $\bfv_k \in \bbR_{\ge 0}^2$ for $k \in \bbZ_{>0}$ with $|\bfv_k| \stackrel{k \to \infty}{\to} \infty$ and $\frac{\bfv_k}{|\bfv_k|} \stackrel{k \to \infty}{\to} \bfv$. Then  
    \begin{align*}
        \frac{1}{|\bfv_k|} \mrL_\zeta(\bfv_k) \stackrel{\rm{a.s.}}{\to} \mrM_{\bfv}(\lambda) \quad \text{ as } k \to \infty.   
    \end{align*}
\end{prop}
\begin{proof}
Write $\bfv = (x,y)$ and $\bfv_k = (x_k,y_k)$. We will prove the theorem assuming without loss of generality that $x_k \ge kx$ and $y_k \ge ky$ for all $k$. These inequalities would hold after rescaling $\bfv_k$ if necessary. 

Introduce the random variables
    \begin{align*}
        \mrI_k &\coloneq \mrL_\zeta(x_k, 0) \sim \Pois[x_k\lambda] \text{ \ and}\\
        \mrJ_k &\coloneq \left(\mrL_\zeta(x_k, y_k) - \mrL_\zeta(x_k, 0)\right) \sim \Pois\left[\frac{y_k}{\lambda(1-t)}\right]
    \end{align*}
where the second line uses the stationarity of $\zeta$. Then $\mrL_\zeta(\bfv_k) = \mrI_k + \mrJ_k$. We recall that if $X \sim \Pois[\mu]$ for some $\mu > 0$ then $\bbE[X] = \mu$. Also, the moment and the cumulant generating functions for the centered variable $X-\mu$ are given by 
\begin{align*}
M(t) = \bbE[\exp\{t(X-\mu)\}] &= \exp\{\mu (e^t-1-t)\} = \sum_{n=0}^\infty \frac{m_nt^n}{n!}, \\ 
K(t) = \log \bbE[\exp\{t(X-\mu)\}] &= \mu (e^t-1-t) = \sum_{n=0}^\infty \frac{\kappa_n t^n}{n!} \quad \text{ for } t \in \bbR
\end{align*}
where $m_n = \bbE[(X-\mu)^n]$ and $\kappa_n = \one_{\{n \ge 2\}}\mu$. For each $n \in \bbZ_{>0}$, computing the $n$th derivative of $M(t) = \exp\{K(t)\}$, one can express the moment $m_n$ in terms of the cumulants $\kappa_1, \dotsc, \kappa_n$. For the fourth moment in particular, one has $m_4 = \kappa_4 + 3\kappa_2^2 = \mu + 3\mu^2$. Therefore, by the fourth moment bound, 
\begin{align*}
\bbP[|X-\mu| \ge a] \le \frac{\bbE[|X-\mu|^4]}{a^4} \le \frac{\mu+3\mu^2}{a^4} \quad \text{ for any } a > 0. 
\end{align*}
Now pick any $\varepsilon > 0$. Applying the preceding bound to $\mrI_k$ and $\mrJ_k$ gives  
\begin{align*}
\bbP\{|I_k-\lambda x_k| \ge \varepsilon |\bfv_k|\} &\le (\lambda x_k + 3\lambda^2 x_k^2) \cdot \frac{1}{\varepsilon^4|\bfv_k|^4}\\
\bbP\bigg\{\bigg|J_k-\frac{y_k}{\lambda(1-t)}\bigg| \ge \varepsilon |\bfv_k|\bigg\} &\le  \bigg(\frac{y_k}{\lambda (1-t)} + \frac{y_k^2}{\lambda^2 (1-t)^2}\bigg) \cdot \frac{1}{\varepsilon^4 |\bfv_k|^4}
\end{align*}
for any $k$. Since $|\bfv_k|^4 = (x_k^2 + y_k^2)^2 \ge k^2 \min\{x, y\}^2(x_k+y_k)^2$, the right-hand sides above are summable in $k$. Because $\varepsilon$ can be taken arbitrarily small, it then follows from Borel--Cantelli that 
\begin{align*}
\frac{1}{|\bfv_k|}(I_k-\lambda x_k) \to 0  \quad \text{ and } \quad \frac{1}{|\bfv_k|}\bigg(J_k-\frac{y_k}{\lambda(1-t)}\bigg) \to 0
\end{align*}
a.s. Since also $\frac{\bfv_k}{|\bfv_k|} \to \bfv$, rearranging the limits above gives  
\begin{align*}
\frac{\mrL_{\zeta}(\bfv_k)}{|\bfv_k|} = \frac{I_k}{|\bfv_k|} + \frac{J_k}{|\bfv_k|} \to \lambda x + \frac{y}{\lambda (1-t)} = \mrM_\bfv(\lambda), 
\end{align*}
which completes the proof. 
\end{proof}

We have described the behaviour of the height function in stationarity. The next lemma provides a similar result for non-stationary 
one-sided processes, that is, with a positive-rate Poisson process of nucleations on one boundary and no nucleations on the other. This lemma will be an ingredient for our local convergence proof in Section \ref{S:PfLocConv}. 
\begin{lem}\label{lm:one-sided-lln}
    Let $\zeta \sim t\text{-\rm{PNG}}(\lambda, 0)$. Fix a point $\bfv=(x,y)$ with $|\bfv|=1$ and let $\bfv_k=(x_k,y_k) \in \bbR_{\ge 0}^2$ for $k \in \bbZ_{>0}$ with $|\bfv_k| \stackrel{k \to \infty}{\to} \infty$ and $\dfrac{\bfv_k}{|\bfv_k|} \stackrel{k \to \infty}{\to} \bfv$. Then
    \begin{align}\label{eq:one-sided-cases}
        \frac{1}{|\bfv_k|}\mrL_\zeta(\bfv_k) \stackrel{\rm{a.s.}}{\to} \begin{cases}
            \gamma_\bfv, & \text{if } \frac{y}{x} \ge \lambda^2(1-t),\\
            \mrM_\bfv(\lambda), & \text{if } \frac{y}{x} \le \lambda^2(1-t).
        \end{cases}
    \end{align}
\end{lem}

We call the line $\frac{y}{x} = \lambda^2(1-t)$ the \textit{characteristic direction (line) of the parameter $\lambda$} by analogy to a similar notion for the exponential corner growth model defined in \cite{Sepp_09_CGM}, since this line goes in the direction for which the values of the shape and mean functions coincide. \cite{Cato_Groe_05} shows that in Hammersley's process, which is equivalent to $0$-PNG, a similarly defined characteristic line with slope $\lambda^2$ is the almost sure limiting direction of an isolated second class particle in the $\lambda$-parameterized stationary version.

\begin{proof}
    First, suppose that $\frac{y}{x} \ge \lambda^2(1-t)$. Let $\theta \sim t\text{-PNG}(0,0)$ be a $t$-PNG process with empty boundaries. Then $\theta$ and $\zeta$ can be coupled so that $\zeta$ can be obtained from $\theta$ by adding sources, and therefore $\mrL_\theta \le \mrL_\zeta$ on $\bbR^2_{\ge 0}$. Also, $\frac{1}{|\bfv_k|}\mrL_\theta(\bfv_k) \stackrel{\rm{a.s.}}{\to} \gamma_\bfv$ due to \cite[Theorem 1.1]{Dril_Lin_24}. It follows that 
    \begin{align}\label{one-sided-lb}
        \gamma_\bfv \le \liminf\limits_{k \to \infty}{\frac{1}{|\bfv_k|}\mrL_\zeta(\bfv_k)} \quad \text{ almost surely}.
    \end{align}
    On the other hand, note that $\lambda_\bfv = \sqrt{\frac{y}{x(1-t)}} \ge \lambda$. Introduce another process
    \[
     \zeta' \sim t\text{-PNG}\Bigl(\lambda_\bfv, \frac{1}{\lambda_\bfv(1-t)}\Bigr),
     \]
     coupled with $\zeta$ so that $\zeta'$ is obtained from $\zeta$ by adding sources and sinks. Then $\mrL_\zeta \le \mrL_{\zeta'}$ on $\bbR^2_{\ge 0}$. To see this, first add in sources, which means extra second class particles on the south boundary (cf.\ Figure \ref{fig:config}). Write \(\mrL_\zeta(\bfv_k)=\mrL_\zeta(0,y_k)+\bigl(\mrL_\zeta(\bfv_k)-\mrL_\zeta(0,y_k)\bigr)\), and notice that these second class particles do not end up on the west boundary, hence the first term is not affected by them. Some may exit the north boundary, adding one to the second term, and some others exit east, not affecting \(\mrL_\zeta(\bfv_k)\) whatsoever. Next, add in sinks, again thought of as second class particles on the west boundary. Write this time \(\mrL(\bfv_k)=\mrL(x_k,0)+\bigl(\mrL(\bfv_k)-\mrL(x_k,0)\bigr)\). These second class particles do not end up on the south boundary, hence the first term is not affected by them. Some exit east, increasing the second term by one, while others exit north, not affecting \(\mrL(\bfv_k)\).

    Also, $\zeta'$ is stationary. Therefore, by Theorem \ref{th:height-lln}, $\frac{1}{|\bfv_k|} \mrL_{\zeta'}(\bfv_k) \stackrel{\rm{a.s.}}{\to} \mrM_{\bfv}(\lambda_\bfv) = \gamma_\bfv$ as $k \to \infty$. It follows that 
    \begin{align}\label{one-sided-ub}
        \limsup\limits_{k \to \infty}{\frac{1}{|\bfv_k|}\mrL_\zeta(\bfv_k)} \le \gamma_\bfv \quad \text{ almost surely}.
    \end{align}
    Combining \eqref{one-sided-lb} and \eqref{one-sided-ub} gives 
    \begin{align*}
        \frac{1}{|\bfv_k|}\mrL_\zeta(\bfv_k) \stackrel{\rm{a.s.}}{\to} \gamma_\bfv \quad \text{ as } k \to \infty.
    \end{align*}
    (Note that when $\frac{y}{x} = \lambda^2(1-t)$, we have $\gamma_\bfv = \mrM_\bfv(\lambda)$ and so \eqref{eq:one-sided-cases} is consistent).
    
    Now suppose that $\frac{y}{x} < \lambda^2(1-t)$. Without loss of generality, we may assume that $\frac{y_k}{x_k} < \lambda^2(1-t)$ for all $k$. Let $\zeta'' \sim t\text{-PNG}(\lambda, \frac{1}{\lambda(1-t)})$. Consider the point $\bfu = (u, v)$ satisfying $\frac{v}{u} = \lambda^2(1-t)$ and $|\bfu|=1$, and the sequence $\bfu_k = (\frac{y_k}{\lambda^2(1-t)}, y_k)$ going to $\infty$ along the characteristic direction of the parameter $\lambda$; we have $\frac{\bfu_k}{|\bfu_k|} \to \bfu$. The limiting height behaviour of this sequence in both $\zeta$ and $\zeta''$ is already known: $\frac{1}{|\bfu_k|}\mrL_\zeta(\bfu_k) \stackrel{\rm{a.s.}}{\to} \gamma_\bfu = \mrM_\bfu(\lambda)$ and $\frac{1}{|\bfu_k|}\mrL_{\zeta''}(\bfu_k) \stackrel{\rm{a.s.}}{\to} \mrM_\bfu(\lambda)$, and therefore $\frac{1}{|\bfu_k|}(\mrL_{\zeta''}(\bfu_k) - \mrL_\zeta(\bfu_k)) \stackrel{\rm{a.s.}}{\to} 0$. Observe that $\mrL_{\zeta''}(\bfu_k) - \mrL_\zeta(\bfu_k)$ and $\mrL_{\zeta''}(\bfv_k) - \mrL_\zeta(\bfv_k)$ are equal to the number of second class particles $\frac{\zeta''}{\zeta}$ crossing the vertical line segments from $\bfu_k$ and $\bfv_k$, respectively, to the bottom boundary, due to the absence of $\frac{\zeta''}{\zeta}$-particles starting from the bottom boundary. But the latter number is at most the former, because $\bfv_k$ is directly to the right from $\bfu_k$, and every $\frac{\zeta''}{\zeta}$-particle passing under $\bfv_k$ must also pass under $\bfu_k$; see Figure \ref{fig:M-convergence-proof}. Hence $\mrL_{\zeta''}(\bfu_k) - \mrL_\zeta(\bfu_k) \ge \mrL_{\zeta''}(\bfv_k) - \mrL_\zeta(\bfv_k) \ge 0$. Using $|\bfu_k| \le |\bfv_k|$ and taking the almost sure limit gives $\frac{1}{|\bfv_k|}(\mrL_{\zeta''}(\bfv_k) - \mrL_\zeta(\bfv_k)) \stackrel{\rm{a.s.}}{\to} 0$. Now, $\frac{1}{|\bfv_k|}\mrL_{\zeta''}(\bfv_k) \stackrel{\rm{a.s.}}{\to} \mrM_\bfv(\lambda)$ implies $\frac{1}{|\bfv_k|}\mrL_\zeta(\bfv_k) \stackrel{\rm{a.s.}}{\to} \mrM_\bfv(\lambda)$.

    \begin{figure}[ht]
        \begin{center}
        \begin{tikzpicture}[scale=0.86]
            \draw(0,0)--(10,0)--(10,8)--(0,8)--cycle;
            \draw[magenta,very thick](0,0)--(10,8);
            \node[magenta] at (7.7,7.5) {$y=x\cdot\lambda^2(1-t)$};
            
            \node[circle,fill,inner sep=1.5pt] at (1.4,0.77) {};
            \node at (1.7,0.57) {$\bfv$};
            \node[circle,fill,inner sep=1.5pt] at (3.3,1.8) {};
            \draw[thick,dashed,cyan](3.3,1.8)--(3.3,0);
            \node at (3.6,1.6) {$\bfv_1$};
            \node[circle,fill,inner sep=1.5pt] at (6,4) {};
            \draw[thick,dashed,cyan](6,4)--(6,0);
            \node at (6.3,3.8) {$\bfv_2$};
            \node[circle,fill,inner sep=1.5pt] at (8.4,6) {};
            \draw[thick,dashed,cyan](8.4,6)--(8.4,0);
            \node at (8.7,5.8) {$\bfv_3$};

            \node[circle,fill,inner sep=1.5pt] at (1.25,1) {};
            \node at (0.95,1.2) {$\bfu$};
            \node[circle,fill,inner sep=1.5pt] at (2.25,1.8) {};
            \draw[thick,dashed,magenta](2.25,1.8)--(2.25,0);
            \node at (1.95,2) {$\bfu_1$};
            \node[circle,fill,inner sep=1.5pt] at (5,4) {};
            \draw[thick,dashed,magenta](5,4)--(5,0);
            \node at (4.7,4.2) {$\bfu_2$};
            \node[circle,fill,inner sep=1.5pt] at (7.5,6) {};
            \draw[thick,dashed,magenta](7.5,6)--(7.5,0);
            \node at (7.2,6.2) {$\bfu_3$};

            \draw[very thick](0,0.3)--(3.1,0.3)--(3.1,2)--(7.7,2)--(7.7,4.5)--(8.8,4.5)--(8.8,5)--(10,5);
            \draw[very thick](0,2.3)--(3.8,2.3)--(3.8,3.4)--(5.4,3.4)--(5.4,5.5)--(10,5.5);
            \draw[very thick](0,5)--(3,5)--(3,5.55)--(5.35,5.55)--(5.35,8);
        \end{tikzpicture}
        \caption{The paths of the second class particles are in black. Note that for each $i$, the dashed line under $\bfv_i$ intersects at most as many black lines as the dashed line under $\bfu_i$.}\label{fig:M-convergence-proof}
        \end{center}
    \end{figure}
\end{proof}

\section{A law of large numbers for second class particles}
\label{S:LLN_2CP}

In this section, we prove the following law of large numbers for the trajectories of second class particles between a stationary process and a process with empty vertical boundary. This result will play an important role in our proof of local convergence in Section \ref{S:PfLocConv}. 

\begin{thm}\label{th:2nd-class-lln}
    Suppose $p$ and $r$ are positive real numbers with $p < r$. Consider two $t$-PNG processes, $\pi \sim t\text{-\rm{PNG}}(p,\frac{1}{p(1-t)})$ and $\rho \sim t\text{-\rm{PNG}}(r, 0)$, coupled increasingly so that $\pi \leq \rho$. Then each $\frac{\rho}{\pi}$-second class particle asymptotically moves with the slope $pr(1-t)$; more precisely, for each $n \in \bbZ$, 
    \begin{align*}
     \frac{\tau}{Q^n_\tau} \stackrel{\rm{a.s.}}{\to} pr(1-t) \quad \text{ as } \tau \to \infty
    \end{align*}
    where $Q^n$ is the particle of $\frac{\rho}{\pi}$ labelled by $n$, and $Q^n_\tau$ is its abscissa at time $\tau$ as defined in Section \ref{S:2CP}.
\end{thm}
\begin{proof}
    We begin by recalling from Section \ref{S:Height} that $\mrL_\pi$ and $\mrL_\rho$ denote the height functions associated with $\pi$ and $\rho$, respectively. For each point $\bfv$ on the non-negative quarter of the unit circle and for each point sequence $\bfv_k$ with $|\bfv_k| \to \infty$ and $\frac{\bfv_k}{|\bfv_k|} \to \bfv$, we have
    \begin{align*}
        \frac{1}{|\bfv_k|}\mrL_\pi(\bfv_k) \stackrel{\rm{a.s.}}{\to} \mrM_\bfv(p)
    \end{align*}
    due to Theorem \ref{th:height-lln} and
    \begin{align*}
        \frac{1}{|\bfv_k|}\mrL_\rho(\bfv_k) \stackrel{\rm{a.s.}}{\to} \begin{cases}
            \gamma_\bfv, & \text{if } \frac{y}{x} \ge r^2(1-t),\\
            \mrM_\bfv(r), & \text{if } \frac{y}{x} \le r^2(1-t)
        \end{cases}
    \end{align*}
    due to Lemma \ref{lm:one-sided-lln}.
    
    Let $\tau_0$ be the time when $Q^n$ appears in the system (note that $\tau_0 = 0$ if $n \ge 0$). For some time $\tau$, consider the rectangle $S_{\bfu_0, \bfu_\tau}$ with opposite corners $\bfu_0=(Q^n_{\tau_0},\tau_0)$ and $\bfu_\tau=(Q^n_\tau,\tau)=(u_{\tau,x},u_{\tau,y})$. For an arbitrary $t$-PNG process, denote by $a_v^\zeta(\tau)$ the number of vertical segments of $\zeta$ that intersect the bottom side of $S_{\bfu_0, \bfu_\tau}$, excluding $\bfu_0$, and whose top end is not on this side (i.e. they have some part strictly above the bottom side of $S_{\bfu_0, \bfu_\tau}$). These segments contribute to the increment of the height function of $\zeta$ between $\bfu_0$ and $(Q^n_\tau, \tau_0)$. Similarly, denote by $a_h^\zeta(\tau)$ the number of horizontal segments of $\zeta$ that intersect the right side of $S_{\bfu_0, \bfu_\tau}$, excluding $(Q^n_\tau, \tau_0)$, and whose left end is not on this side. These segments contribute to the height increment between $(Q^n_\tau, \tau_0)$ and $\bfu_\tau$ if $\zeta$ is $t$-PNG. Then 
    \begin{align*}
        \mrL_\pi(\bfu_\tau) - \mrL_\pi(\bfu_0) = a_v^\pi(\tau) + a_h^\pi(\tau),\\
        \mrL_\rho(\bfu_\tau) - \mrL_\rho(\bfu_0) = a_v^\rho(\tau) + a_h^\rho(\tau).
    \end{align*}
    Subtracting the first line above from the second gives 
    \begin{align*}
        \mrL_\rho(\bfu_\tau) - \mrL_\pi(\bfu_\tau) + (\mrL_\pi(\bfu_0) -\mrL_\rho(\bfu_0)) &= (a_v^\rho(\tau) - a_v^\pi(\tau)) + (a_h^\rho(\tau) - a_h^\pi(\tau)) \\
        &= a_v^{\rho/\pi}(\tau) - a_h^{\rho/\pi}(\tau).
    \end{align*}
    The trajectories of second class particles do not cross each other. Hence, every $\frac{\rho}{\pi}$-particle other than $Q^n$ that enters the interior of $S_{\bfu_0, \bfu_\tau}$ through the bottom side exits through the right side, and vice versa. $Q^n$ is counted in $a_v^{\rho/\pi}(\tau)$ if it passes some positive distance along the bottom side of $S_{\bfu_0, \bfu_\tau}$ before turning up, and is counted in $a_h^{\rho/\pi}(\tau)$ if it first arrives to the right side later than at time $\tau_0$. Therefore, $a_v^{\rho/\pi}(\tau) - a_h^{\rho/\pi}(\tau) \in \{0, \pm1\}$. Hence, 
    \begin{align}\label{const-dif}
        |\mrL_\rho(\bfu_\tau) - \mrL_\pi(\bfu_\tau)| \le |\mrL_\rho(\bfu_0) -\mrL_\pi(\bfu_0)| + 1,
    \end{align}
    which is a finite number and does not depend on $\tau$ (more precisely, $|\mrL_\rho(\bfu_0) -\mrL_\pi(\bfu_0)| = n$ for $n \ge 1$ and $|\mrL_\rho(\bfu_0) -\mrL_\pi(\bfu_0)| = |n| + 1$ for $n \le 0$). See Figure \ref{fig:second-class-path} for an example.

\begin{figure}[ht]
 \begin{center}
  \begin{tikzpicture}[scale=0.86]
   \draw[black,very thick,preaction={draw,yellow,-,double=yellow,double distance=2\pgflinewidth,}](0,0)--(0,0.9)--(3.6,0.9)--(3.6,2.4)--(8,2.4)--(8,4);
   \draw[black,very thick](5,0)--(5,1.1)--(8,1.1);
   \draw[black,very thick](0,1.6)--(2.4,1.6)--(2.4,2.4)--(3.44,2.4)--(3.44,4);
   
   \draw(0,0)--(0,4)--(8,4)--(8,0)--cycle;
   \draw[red,very thick](0,0.98)--(3.44,0.98)--(3.44,0);
   \draw[red,very thick](7.84,0)--(7.84,2.48)--(0,2.48);
   \draw[red,very thick](8,1.18)--(4.6,1.18)--(4.6,4);
   \draw[red,very thick](0,1.68)--(2.24,1.68)--(2.24,0);

   \draw[blue,very thick](7.92,0)--(7.92,4);
   \draw[blue,very thick](4.92,0)--(4.92,1.1)--(4.52,1.1)--(4.52,4);
   \draw[blue,very thick](3.52,0)--(3.52,4);
   \draw[blue,very thick](2.32,0)--(2.32,2.4)--(0,2.4);
   \draw[blue,very thick](0.08,0)--(0.08,0.82)--(0,0.82);

   \node[circle,fill,inner sep=1.5pt] at (0,0) {};
   \node[anchor=base] at (-0.4,-0.5) {$\bfu_0$:};
   \node[red,anchor=base] at (0,-0.5) {0};
   \node[blue,anchor=base] at (0.3,-0.5) {0};
   \node[circle,fill,inner sep=1.5pt] at (8,4) {};
   \node[anchor=base] at (7.6,4.3) {$\bfu_\tau$:};
   \node[red,anchor=base] at (8,4.3) {4};
   \node[blue,anchor=base] at (8.3,4.3) {4};
   \node[circle,fill,inner sep=1.5pt] at (8,0) {};
   \node[anchor=base] at (8,-0.5) {$(Q^n_\tau, \tau_0)$};

   \node[red] at (0.9,0.4) {\small0};
   \node[red] at (2.8,0.4) {\small1};
   \node[red] at (4.1,0.4) {\small2};
   \node[red] at (6.3,0.4) {\small2};
   \node[red] at (0.9,1.3) {\small1};
   \node[red] at (2.8,1.7) {\small2};
   \node[red] at (0.9,2) {\small2};
   \node[red] at (0.9,3.2) {\small3};
   \node[red] at (6,1.7) {\small3};
   \node[red] at (3.9,3.2) {\small3};
   \node[red] at (6,3.2) {\small4};

   \node[blue] at (1.2,0.4) {\small0};
   \node[blue] at (3.1,0.4) {\small1};
   \node[blue] at (4.4,0.4) {\small2};
   \node[blue] at (6.6,0.4) {\small3};
   \node[blue] at (1.2,1.3) {\small0};
   \node[blue] at (3.1,1.7) {\small1};
   \node[blue] at (1.2,2) {\small0};
   \node[blue] at (1.2,3.2) {\small1};
   \node[blue] at (6.3,1.7) {\small3};
   \node[blue] at (4.2,3.2) {\small2};
   \node[blue] at (6.3,3.2) {\small3};
  \end{tikzpicture}
  \caption{Possible configuration inside $S_{\bfu_0, \bfu_\tau}$. $Q^n$ is highlighted with yellow. Red and blue numbers denote relative heights with respect to $\bfu_0$ of $\pi$ and $\rho$, respectively. Note that the difference between the heights of $\pi$ and $\rho$ along the path of $Q^n$ always differs at most by $1$ from such difference at $\bfu_0$.}\label{fig:second-class-path}
 \end{center}
\end{figure}
    
    Pick some divergent increasing sequence $(\tau_i)_{i \in \bbZ_{>0}}$ of times for which $\frac{\bfu_{\tau_i}}{|\bfu_{\tau_i}|} \stackrel{i \to \infty}{\to} \bfu$ for some point $\bfu = (u_x,u_y)$ on the unit circle; this is possible by choosing any increasing divergent sequence and then choosing its subsequence that satisfies needed convergence. Then, both $\frac{1}{|\bfu_{\tau_i}|}\mrL_\pi(\bfu_{\tau_i})$ and $\frac{1}{|\bfu_{\tau_i}|}\mrL_\rho(\bfu_{\tau_i})$ converge in $\bbR$ almost surely. But \eqref{const-dif} implies that $\frac{1}{|\bfu_{\tau_i}|}(\mrL_\rho(\bfu_{\tau_i}) - \mrL_\pi(\bfu_{\tau_i})) \stackrel{\rm{a.s.}}{\to} 0$, therefore $\lim\limits_{i \to \infty}{\frac{1}{|\bfu_{\tau_i}|}\mrL_\pi(\bfu_{\tau_i})} \stackrel{\rm{a.s.}}{=} \lim\limits_{i \to \infty}{\frac{1}{|\bfu_{\tau_i}|}\mrL_\rho(\bfu_{\tau_i})}$, given that these almost sure limits exist. From the values of these limits, we then conclude that 
    \begin{align*}
        \mrM_\bfu(p) = \begin{cases}
            \gamma_\bfu, & \text{if } \frac{u_y}{u_x} \ge r^2(1-t);\\
            \mrM_\bfu(r), & \text{if } \frac{u_y}{u_x} \le r^2(1-t).
        \end{cases}
    \end{align*}
    But the equality $\mrM_\bfu(p) = \gamma_\bfu$ can only hold when $\bfu$ is on the characteristic line of parameter $p$, that is, $\frac{u_y}{u_x} = p^2(1-t) < r^2(1-t)$. Hence, only the second case is possible: $\mrM_\bfu(p) = \mrM_\bfu(r)$. This means that  
    \begin{align*}
        u_xp + \frac{u_y}{p(1-t)} = u_xr + \frac{u_y}{r(1-t)},
    \end{align*}
    that is, $\frac{u_y}{u_x} = pr(1-t)$, 
    which is indeed less than $r^2(1-t)$.

    Hence, for every divergent increasing sequence $(\tau_i)_{i \in \bbZ_{>0}}$ of times for which $\frac{u_{\tau_i,y}}{u_{\tau_i,x}}$ converges in $\bbR^2_{\ge 0}$, the limit is $pr(1-t)$ a.s. Hence, for every divergent increasing $(\tau_i)_{i \in \bbZ_{>0}}$, the convergence $\frac{u_{\tau_i,y}}{u_{\tau_i,x}} \stackrel{\rm{a.s.}}{\to} pr(1-t)$ holds.
\end{proof}

\section{Proof of local convergence}
\label{S:PfLocConv}

We are now ready to implement the plan outlined after Theorem \ref{th:LocConv}. For an arbitrary $\varepsilon > 0$, we consider the following three $t$-PNG processes: 
\begin{itemize}
    \item $\alpha \sim t\text{-PNG}\left(\lambda, \dfrac{1}{\lambda(1-t)}\right)$,
    \item $\eta = \eta(\varepsilon) \sim t\text{-PNG}\left(\lambda+\varepsilon, \dfrac{1}{(\lambda+\varepsilon)(1-t)}\right)$,
    \item $\omega = \omega(\varepsilon) \sim t\text{-PNG}(\lambda+\varepsilon, 0)$.
\end{itemize}

\begin{figure}[ht]
 \begin{center}
  \begin{tikzpicture}[scale=0.5]
  \begin{scope}
   \draw(0,0)--(12,0)--(12,8)--(0,8)--cycle;
   \node[red] at (3, 0) {$\times$};
   \node[red] at (5, 0) {$\times$};
   \node[red] at (8, 0) {$\times$};
   \node[red] at (11.5, 0) {$\times$};
   
   \node[red] at (0, 1) {$\times$};
   \node[red] at (0, 3) {$\times$};
   \node[red] at (0, 4) {$\times$};
   \node[red] at (0, 6) {$\times$};
   \node[red] at (0, 7.5) {$\times$};

   \node at (2, 5) {$\times$};
   \node at (2.4, 2.5) {$\times$};
   \node at (2.8, 6.9) {$\times$};
   \node at (4, 0.5) {$\times$};
   \node at (5.8, 2) {$\times$};
   \node at (6, 3.8) {$\times$};
   \node at (7, 7) {$\times$};
   \node at (8.5, 4.5) {$\times$};
   \node at (9.5, 1.4) {$\times$};
   \node at (11, 5.5) {$\times$};

   \draw[red,very thick](12,0.5)--(4,0.5)--(4,5)--(2,5)--(2,6)--(0,6);
   \draw[red,very thick](12,2)--(5.8,2)--(5.8,2.5)--(2.4,2.5)--(2.4,4)--(0,4);
   \draw[red,very thick](8,0)--(8,3.8)--(6,3.8)--(6,6.9)--(2.8,6.9)--(2.8,7.5)--(0,7.5);
   \draw[red,very thick](3,0)--(3,1)--(0,1);
   \draw[red,very thick](5,0)--(5,3)--(0,3);
   \draw[red,very thick](11.5,0)--(11.5,1.4)--(9.5,1.4)--(9.5,4.5)--(8.5,4.5)--(8.5,7)--(7,7)--(7,8);
   \draw[red,very thick](12,5.5)--(11,5.5)--(11,8);
   \end{scope}
   \begin{scope}[xshift=14cm]
       \draw(0,0)--(12,0)--(12,8)--(0,8)--cycle;
   \node[cyan] at (1.5, 0) {$\times$};
   \node[cyan] at (3, 0) {$\times$};
   \node[cyan] at (5, 0) {$\times$};
   \node[cyan] at (6.3, 0) {$\times$};
   \node[cyan] at (8, 0) {$\times$};
   \node[cyan] at (10, 0) {$\times$};
   \node[cyan] at (11.5, 0) {$\times$};
   
   \node at (2, 5) {$\times$};
   \node at (2.4, 2.5) {$\times$};
   \node at (2.8, 6.9) {$\times$};
   \node at (4, 0.5) {$\times$};
   \node at (5.8, 2) {$\times$};
   \node at (6, 3.8) {$\times$};
   \node at (7, 7) {$\times$};
   \node at (8.5, 4.5) {$\times$};
   \node at (9.5, 1.4) {$\times$};
   \node at (11, 5.5) {$\times$};

   \draw[cyan,very thick](1.5,0)--(1.5,8);
   \draw[cyan,very thick](6.3,0)--(6.3,0.5)--(4,0.5)--(4,8);
   \draw[cyan,very thick](12,2)--(5.8,2)--(5.8,2.5)--(2.4,2.5)--(2.4,8);
   \draw[cyan,very thick](8,0)--(8,3.8)--(6,3.8)--(6,8);
   \draw[cyan,very thick](3,0)--(3,5)--(2,5)--(2,8);
   \draw[cyan,very thick](5,0)--(5,6.9)--(2.8,6.9)--(2.8,8);
   \draw[cyan,very thick](11.5,0)--(11.5,1.4)--(9.5,1.4)--(9.5,4.5)--(8.5,4.5)--(8.5,7)--(7,7)--(7,8);
   \draw[cyan,very thick](12,5.5)--(11,5.5)--(11,8);
   \draw[cyan,very thick](10,0)--(10,8);

   \draw[black,very thick](0,1)--(3,1)--(3,5)--(4,5)--(4,8);
   \draw[black,very thick](0,3)--(5,3)--(5,6.9)--(6,6.9)--(6,8);
   \draw[black,very thick](0,6)--(1.5,6)--(1.5,7.5)--(2,7.5)--(2.8,7.5)--(2.8,8);
   \draw[black,very thick](1.5,0)--(1.5,6)--(2,6)--(2,8);
   \draw[black,very thick](6.3,0)--(6.3,0.5)--(12,0.5);
   \draw[black,very thick](10,0)--(10,8);
   \draw[black,very thick](0,4)--(2.4,4)--(2.4,8);
   \draw[black,very thick](0,7.5)--(1.5,7.5)--(1.5,8);

   \node[black] at (-0.5, 7.5) {\small-4};
   \node[black] at (-0.5, 6) {\small-3};
   \node[black] at (-0.5, 4) {\small-2};
   \node[black] at (-0.5, 3) {\small-1};
   \node[black] at (-0.5, 1) {\small0};
   \node[black] at (1.5, -0.5) {\small1};
   \node[black] at (6.3, -0.5) {\small2};
   \node[black] at (10, -0.5) {\small3};
   \end{scope}
   \begin{scope}[yshift=-10cm]
       \draw(0,0)--(12,0)--(12,8)--(0,8)--cycle;
   \node[blue] at (1.5, 0) {$\times$};
   \node[blue] at (3, 0) {$\times$};
   \node[blue] at (5, 0) {$\times$};
   \node[blue] at (6.3, 0) {$\times$};
   \node[blue] at (8, 0) {$\times$};
   \node[blue] at (10, 0) {$\times$};
   \node[blue] at (11.5, 0) {$\times$};
   
   \node[blue] at (0, 1) {$\times$};
   \node[blue] at (0, 3) {$\times$};
   \node[blue] at (0, 6) {$\times$};
   
   \node at (2, 5) {$\times$};
   \node at (2.4, 2.5) {$\times$};
   \node at (2.8, 6.9) {$\times$};
   \node at (4, 0.5) {$\times$};
   \node at (5.8, 2) {$\times$};
   \node at (6, 3.8) {$\times$};
   \node at (7, 7) {$\times$};
   \node at (8.5, 4.5) {$\times$};
   \node at (9.5, 1.4) {$\times$};
   \node at (11, 5.5) {$\times$};

   \draw[blue,very thick](1.5,0)--(1.5,6)--(0,6);
   \draw[blue,very thick](6.3,0)--(6.3,0.5)--(4,0.5)--(4,5)--(2,5)--(2,8);
   \draw[blue,very thick](12,2)--(5.8,2)--(5.8,2.5)--(2.4,2.5)--(2.4,8);
   \draw[blue,very thick](8,0)--(8,3.8)--(6,3.8)--(6,6.9)--(2.8,6.9)--(2.8,8);
   \draw[blue,very thick](3,0)--(3,1)--(0,1);
   \draw[blue,very thick](5,0)--(5,3)--(0,3);
   \draw[blue,very thick](11.5,0)--(11.5,1.4)--(9.5,1.4)--(9.5,4.5)--(8.5,4.5)--(8.5,7)--(7,7)--(7,8);
   \draw[blue,very thick](12,5.5)--(11,5.5)--(11,8);
   \draw[blue,very thick](10,0)--(10,8);
   \end{scope}
   \begin{scope}[xshift=14cm,yshift=-10cm]
       \draw(0,0)--(12,0)--(12,8)--(0,8)--cycle;
   \node[cyan] at (1.5, 0) {$\times$};
   \node[cyan] at (3, 0) {$\times$};
   \node[cyan] at (5, 0) {$\times$};
   \node[cyan] at (6.3, 0) {$\times$};
   \node[cyan] at (8, 0) {$\times$};
   \node[cyan] at (10, 0) {$\times$};
   \node[cyan] at (11.5, 0) {$\times$};
   
   \node at (2, 5) {$\times$};
   \node at (2.4, 2.5) {$\times$};
   \node at (2.8, 6.9) {$\times$};
   \node at (4, 0.5) {$\times$};
   \node at (5.8, 2) {$\times$};
   \node at (6, 3.8) {$\times$};
   \node at (7, 7) {$\times$};
   \node at (8.5, 4.5) {$\times$};
   \node at (9.5, 1.4) {$\times$};
   \node at (11, 5.5) {$\times$};

   \draw[cyan,very thick](1.5,0)--(1.5,8);
   \draw[cyan,very thick](6.3,0)--(6.3,0.5)--(4,0.5)--(4,8);
   \draw[cyan,very thick](12,2)--(5.8,2)--(5.8,2.5)--(2.4,2.5)--(2.4,8);
   \draw[cyan,very thick](8,0)--(8,3.8)--(6,3.8)--(6,8);
   \draw[cyan,very thick](3,0)--(3,5)--(2,5)--(2,8);
   \draw[cyan,very thick](5,0)--(5,6.9)--(2.8,6.9)--(2.8,8);
   \draw[cyan,very thick](11.5,0)--(11.5,1.4)--(9.5,1.4)--(9.5,4.5)--(8.5,4.5)--(8.5,7)--(7,7)--(7,8);
   \draw[cyan,very thick](12,5.5)--(11,5.5)--(11,8);
   \draw[cyan,very thick](10,0)--(10,8);

   \draw[black,very thick](0,1)--(3,1)--(3,5)--(4,5)--(4,8);
   \draw[black,very thick](0,3)--(5,3)--(5,6.9)--(6,6.9)--(6,8);
   \draw[black,very thick](0,6)--(1.5,6)--(1.5,8);

   \node[black] at (-0.5, 6) {\small-2};
   \node[black] at (-0.5, 3) {\small-1};
   \node[black] at (-0.5, 1) {\small0};
   \end{scope}
   \begin{scope}[xshift=7cm,yshift=-20cm]
       \draw(0,0)--(12,0)--(12,8)--(0,8)--cycle;
   \node[cyan] at (1.5, 0) {$\times$};
   \node[cyan] at (3, 0) {$\times$};
   \node[cyan] at (5, 0) {$\times$};
   \node[cyan] at (6.3, 0) {$\times$};
   \node[cyan] at (8, 0) {$\times$};
   \node[cyan] at (10, 0) {$\times$};
   \node[cyan] at (11.5, 0) {$\times$};
   
   \node at (2, 5) {$\times$};
   \node at (2.4, 2.5) {$\times$};
   \node at (2.8, 6.9) {$\times$};
   \node at (4, 0.5) {$\times$};
   \node at (5.8, 2) {$\times$};
   \node at (6, 3.8) {$\times$};
   \node at (7, 7) {$\times$};
   \node at (8.5, 4.5) {$\times$};
   \node at (9.5, 1.4) {$\times$};
   \node at (11, 5.5) {$\times$};

   \draw[cyan,very thick](1.5,0)--(1.5,8);
   \draw[cyan,very thick](6.3,0)--(6.3,0.5)--(4,0.5)--(4,8);
   \draw[cyan,very thick](12,2)--(5.8,2)--(5.8,2.5)--(2.4,2.5)--(2.4,8);
   \draw[cyan,very thick](8,0)--(8,3.8)--(6,3.8)--(6,8);
   \draw[cyan,very thick](3,0)--(3,5)--(2,5)--(2,8);
   \draw[cyan,very thick](5,0)--(5,6.9)--(2.8,6.9)--(2.8,8);
   \draw[cyan,very thick](11.5,0)--(11.5,1.4)--(9.5,1.4)--(9.5,4.5)--(8.5,4.5)--(8.5,7)--(7,7)--(7,8);
   \draw[cyan,very thick](12,5.5)--(11,5.5)--(11,8);
   \draw[cyan,very thick](10,0)--(10,8);
   \end{scope}
  \end{tikzpicture}
 \end{center}
\caption{Possible configurations. Second class processes $\frac{\omega}{\alpha}$ and $\frac{\omega}{\eta}$ are in black and indexed. Top row: $\alpha$ and $\frac{\omega}{\alpha}$. Middle row: $\eta$ and $\frac{\omega}{\eta}$. Bottom row: $\omega$.}\label{fig:alpha-eta-omega}
\end{figure}

In particular, $\eta$ and $\alpha$ are stationary. As in Corollary \ref{cor:poi-coupling}, we can couple the processes so that $\alpha \le \eta \le \omega$. We will analyse two processes of second class particles, $\frac{\omega}{\eta}$ and $\frac{\omega}{\alpha}$. Recall that particles in both processes are indexed by integers; see the end of Section \ref{S:2CP}. See also Figure \ref{fig:alpha-eta-omega} for an example. 

Due to the assumed coupling, the union of trajectories of $\frac\omega\eta$-particles is a subset of the union of trajectories of $\frac\omega\alpha$-particles. Also, $(\frac{\omega}{\eta})_\uparrow = \varnothing$ since $\omega_\uparrow = \eta_\uparrow$; in other words, all $\frac{\omega}{\eta}$-particles start their movement on the left boundary and hence are indexed by non-positive integers. Fix some index $k \in \bbZ_{\le 0}$. For any $\sigma \ge 0$, write $X^k(\sigma)$ for the index of the $\frac{\omega}{\alpha}$-particle that coincides with the $k$-th $\frac{\omega}{\eta}$-particle. We can define $X^k(\sigma)$ to be right-continuous in $\sigma$.

Now consider the process $\underline{V} = (V_j)_{j \in \bbZ} \colon \bbR_{\ge 0} \mapsto \{0, 1\}^\bbZ$ given by 
\begin{align}
V_j(\sigma) = \one\{j = X^k(\sigma) \text{ for some } k \in \bbZ_{\le 0}\}.      
\label{V-i-definition}\end{align}
That is, $V_j(\sigma)$ indicates whether the $j$-th $\frac{\omega}{\alpha}$-particle coincides with some $\frac{\omega}{\eta}$-particle at time $\sigma$. To describe the evolution of $\underline{V}$, we first note that the set of all times when two $\frac{\omega}{\alpha}$-particles meet a.s.\ can be ordered to form a strictly increasing positive sequence $(\sigma_i)_{i \in \bbZ_{>0}}$. Setting $\sigma_0 = 0$, we introduce the discrete-time process given by $\underline{V}^i = \underline{V}(\sigma_i)$ for $i \in \bbZ_{\ge 0}$. The initial state for this process is as follows. For $j > 0$, $V^0_j = 0$ because all $\frac{\omega}{\eta}$- particles start from the left boundary. Also, since $\eta_\rightarrow$ is an $\frac{\lambda}{\lambda+\varepsilon}$-thinning of $\alpha_\rightarrow$ on the left boundary, $V^0_j \sim \Ber[\frac{\lambda}{\lambda+\varepsilon}]$ for $j \le 0$, and these random variables are jointly independent. 

With probability one, at each time $\sigma_i$ exactly one meeting occurs between $\frac{\omega}{\alpha}$-particles; let these two meeting particles have indices $\mu_i$ and $\mu_i+1$. These are the only two indices where $\underline{V}^i$ can differ from $\underline{V}^{i-1}$. The following lemma describes this evolution.
\begin{lem}\label{jump-rules}
Conditionally on $\omega$ and $\alpha$, the process $\underline{V}$ evolves according to the following rules for each $i \in \bbZ_{>0}$: 
\begin{enumerate}[\normalfont (1)]
    \item if $V^{i-1}_{\mu_i} = V^{i-1}_{\mu_i+1}$ then $\underline{V}^i = \underline{V}^{i-1}$; 
    \item if $V^{i-1}_{\mu_i} = 0$ and $V^{i-1}_{\mu_i+1} = 1$ then $\underline{V}^i = \underline{V}^{i-1} - e_{\mu_i+1} + e_{\mu_i}$ with probability 1 (a left jump);
    \item if $V^{i-1}_{\mu_i} = 1$ and $V^{i-1}_{\mu_i+1} = 0$ then $\underline{V}^i = \underline{V}^{i-1} - e_{\mu_i} + e_{\mu_i+1}$ with probability $t$ (a right jump), and $\underline{V}^i = \underline{V}^{i-1}$ with probability $1-t$,
\end{enumerate}
where $e_\mu: \bbZ \to \{0, 1\}$ denotes the function that equals $1$ at $\mu$, and $0$ elsewhere. 
\end{lem}
\begin{proof}
    Before meeting at time $\sigma_i$, one $\frac{\omega}{\alpha}$-particle moves from left to right, following a horizontal line segment present in $\alpha$ but not in $\omega$, and another one moves from bottom to top, following a vertical line segment present in $\omega$ but not in $\alpha$. Let $\bfP_i$ denote the meeting point of these two $\frac{\omega}{\alpha}$-particles. Hence, $\bfP_i$ lies inside a vertical $\omega$-segment.
    
    Case (1) corresponds to the situation when both meeting $\frac{\omega}{\alpha}$-particles either carry no $\frac{\omega}{\eta}$-particles, in which case $\underline{V}$ stays the same (Figure \ref{fig:ahw-turns}.a), or carry one each, and hence both these $\frac{\omega}{\eta}$-particles are blocked from switching to the neighbouring $\frac{\omega}{\alpha}$-particles (Figure \ref{fig:ahw-turns}.b).
    
    In case (2), only the vertically moving $\frac{\omega}{\alpha}$-particle carries a $\frac{\omega}{\eta}$-particle. There are no horizontal lines through $\bfP_i$ neither in $\omega$ nor in $\eta$, hence the $\frac{\omega}{\eta}$-particle continues to move up after passing through $\bfP_i$, switching its associated $\frac{\omega}{\alpha}$-particle to the previous one (Figure \ref{fig:ahw-turns}.c).
    
    In case (3), only the horizontally moving $\frac{\omega}{\alpha}$-particle carries a $\frac{\omega}{\eta}$-particle. At point $\bfP_i$, this $\frac{\omega}{\eta}$-particle, which moves along a horizontal $\eta$-line, meets a vertical $\eta$-line. This corresponds to a sampling of a crossing or corner point of $\eta$ at $\bfP_i$. A crossing point is formed with probability $t$, so that the $\frac{\omega}{\eta}$-particle continues to move to the right and hence switches to the next $\frac{\omega}{\alpha}$-particle (Figure \ref{fig:ahw-turns}.d). Alternatively, a corner point is formed with probability $1-t$, with the $\frac{\omega}{\eta}$-particle turning up and staying with the same $\frac{\omega}{\alpha}$-particle (Figure \ref{fig:ahw-turns}.e).
\end{proof}

\begin{figure}[ht]
 \begin{center}
  \begin{tikzpicture}[scale=0.69]
   \begin{scope}[xshift=0cm]
    \draw(0,0)--(4,0)node[midway,below=8pt]{(a)}--(4,4)--(0,4)--cycle;
    \draw[red,very thick](4,2)--(0,2);
    \draw[black,very thick](2,0)--(2,4);
    \draw[blue,very thick](2.2,0)--(2.2,4);
    \draw[red,very thick,dashed](2.3,0)--(2.3,1.9)--(4,1.9);
    \draw[red,very thick,dashed](0,2.1)--(1.9,2.1)--(1.9,4);
   \end{scope}
   \begin{scope}[xshift=6cm]
    \draw(0,0)--(4,0)node[midway,below=8pt]{(b)}--(4,4)--(0,4)--cycle;
    \draw[red,very thick](4,2)--(0,2);
    \draw[blue,very thick](4,1.8)--(0,1.8);
    \draw[black,very thick](2,0)--(2,4);
    \draw[red,very thick,dashed](2.1,0)--(2.1,1.7)--(4,1.7);
    \draw[red,very thick,dashed](0,2.1)--(1.9,2.1)--(1.9,4);
    \draw[blue,very thick,dashed](2.2,0)--(2.2,1.6)--(4,1.6);
    \draw[blue,very thick,dashed](0,2.2)--(1.8,2.2)--(1.8,4);
   \end{scope}
   \begin{scope}[xshift=12cm]
    \draw(0,0)--(4,0)node[midway,below=8pt]{(c)}--(4,4)--(0,4)--cycle;
    \draw[red,very thick](4,2)--(0,2);
    \draw[black,very thick](2,0)--(2,4);
    \draw[blue,very thick,dashed](2.1,0)--(2.1,4);
    \draw[red,very thick,dashed](2.2,0)--(2.2,1.9)--(4,1.9);
    \draw[red,very thick,dashed](0,2.1)--(1.9,2.1)--(1.9,4);
   \end{scope}
   \begin{scope}[xshift=3cm,yshift=-6cm]
    \draw(0,0)--(4,0)node[midway,below=8pt]{(d)}--(4,4)--(0,4)--cycle;
    \draw[red,very thick](4,2)--(0,2);
    \draw[black,very thick](2,0)--(2,4);
    \draw[blue,very thick](1.8,0)--(1.8,4);
    \draw[red,very thick,dashed](2.1,0)--(2.1,1.7)--(4,1.7);
    \draw[red,very thick,dashed](0,2.1)--(1.7,2.1)--(1.7,4);
    \draw[blue,very thick](4,1.8)--(0,1.8);
    \draw[blue,very thick,dashed](0,1.9)--(4,1.9);
   \end{scope}
   \begin{scope}[xshift=9cm,yshift=-6cm]
    \draw(0,0)--(4,0)node[midway,below=8pt]{(e)}--(4,4)--(0,4)--cycle;
    \draw[red,very thick](4,2)--(0,2);
    \draw[black,very thick](2,0)--(2,4);
    \draw[red,very thick,dashed](2.1,0)--(2.1,1.9)--(4,1.9);
    \draw[red,very thick,dashed](0,2.1)--(1.9,2.1)--(1.9,4);
    \draw[blue,very thick,dashed](0,1.9)--(1.8,1.9)--(1.8,4);
    \draw[blue,very thick](0,1.8)--(1.8,1.8)--(1.8,0);
   \end{scope}
  \end{tikzpicture}
  \caption{Possible configurations at point $\bfP_i$ (in the middle). We use red for $\alpha$, blue for $\eta$, black for $\omega$, red dashed for $\frac{\omega}{\alpha}$, blue dashed for $\frac{\omega}{\eta}$.}\label{fig:ahw-turns}
 \end{center}
\end{figure}

Recall that, as defined above \eqref{V-i-definition}, $X^0(\sigma)$ is the index of the $\frac{\omega}{\alpha}$-particle that coincides with the rightmost $\frac{\omega}{\eta}$-particle at time $\sigma$. The next lemma provides an exponentially decaying bound on the right tail of the label $X^0(\sigma)$ at any time $\sigma \ge 0$. 
The following idea is similar to the one used in Ferrari-Kipnis-Saada \cite{fks_91} and Bal\'azs-Sepp\"al\"ainen \cite[Lemma 4.1]{Bala_Sepp_09}.

\begin{lem}\label{RightTail}
There exist constants $s \in (0, 1)$ and $N \in \bbZ_{>0}$ such that $\bbP\{ X^0(\sigma) \ge n\} \le s^n$ for $n \ge N$ and $\sigma \ge 0$. 
\end{lem}
\begin{proof}
    We will construct a $\{0, 1\}$-valued, discrete-time stationary process $(\underline{U}^i: i \in \bbZ_{\ge 0})$ coupled with $(\underline{V}^i: i \in \bbZ_{\ge 0})$ such that $\underline{U}^i \ge \underline{V}^i$ for every $i$. The transition probabilities for $\underline{U}^i$ will be the same as those of $\underline{V}^i$, which were described in Lemma \ref{jump-rules}. The distribution of the initial state $\underline{U}^0$ will be chosen as a \emph{product blocking measure}, that is, all coordinates of $\underline{U}^0$ are independent and, 
    \begin{align*}
    U^0_j \sim \Ber\bigg[\frac{t^{j+c}}{1+t^{j+c}}\bigg]    
    \end{align*}
    for every $j \in \bbZ$ and some constant $c \in \bbZ$. 
    The following calculation verifies the reversibility of the $\underline{U}$-process. For every $i \in \bbZ$, (assuming $\underline{U}^{i-1}$ is distributed as above), 
    \begin{alignat*}{2}
        &\bbP\{ \text{ jump from $\mu_i$ to $\mu_i+1$ at time $\sigma_i$ } \} \\ 
        &= t\cdot\bbP\{ U^{i-1}_{\mu_i} = 1 \text{ and } U^{i-1}_{\mu_i+1} = 0 \}\\
        &= t\cdot\frac{t^{\mu_i+c}}{1+t^{\mu_i+c}}\cdot\frac{1}{1+t^{\mu_i+1+c}} \\ 
        &= 1\cdot\frac{1}{1+t^{\mu_i+c}}\cdot\frac{t^{\mu_i+1+c}}{1+t^{\mu_i+1+c}}\\
        &= 1\cdot\bbP\{ U^{i-1}_{\mu_i} = 0 \text{ and } U^{i-1}_{\mu_i+1} = 1 \} \\ 
        &= \bbP\{ \text{ jump from $\mu_i+1$ to $\mu_i$ at time $\sigma_i$ } \}.
    \end{alignat*}
    Hence $\underline{U}^i$ is stationary. Choose $c = \log_t\frac{\lambda}{\varepsilon}$. Then $\frac{t^c}{1+t^c} = \frac{\lambda}{\lambda+\varepsilon}$ and $\frac{t^{j+c}}{1+t^{j+c}} > \frac{\lambda}{\lambda+\varepsilon}$ for all $j < 0$. This means $U^0_j$ stochastically dominates $V^0_j$ for every $j$. Hence $\underline{U}^0$ can be coupled with $\underline{V}^0$ so that $\underline{U}^0 \ge \underline{V}^0$. 
    
    We will now describe the extension of this coupling for $\underline{U}^i$ and $\underline{V}^i$ with $i > 0$, in which the inequality $\underline{U}^i \ge \underline{V}^i$ persists for all $i$. We shall aim for $\underline{U}^i \ge \underline{V}^i$ assuming $\underline{U}^{i-1} \ge \underline{V}^{i-1}$. Recall that the only possible changes between $\underline{V}^{i-1}$ and $\underline{V}^i$, and therefore between $\underline{U}^{i-1}$ and $\underline{U}^i$, are jumps between positions $\mu_i$ and $\mu_i+1$. These are the possible cases (two bits stand for the entries in positions $\mu_i$ and $\mu_i+1$ of respective random vectors). 
    
    \begin{align*}
    \begin{cases}
        \underline{V}^{i-1} : 00, \ \begin{cases}
            \underline{U}^{i-1} : 00 \rightarrow \hspace{8.10pt}\underline{V}^i : 00, \ \underline{U}^i : 00\\
            \underline{U}^{i-1} : 01 \rightarrow \hspace{8.10pt}\underline{V}^i : 00, \ \underline{U}^i : 10 \text{ (left jump in $\underline{U}$)}\\
            \underline{U}^{i-1} : 10 \rightarrow \begin{cases} 
                \underline{V}^i : 00, \ \underline{U}^i : 01 \text{ with probability $t$ (right jump in $\underline{U}$)}\\
                \underline{V}^i : 00, \ \underline{U}^i : 10 \text{ with probability $1-t$}
            \end{cases}\\
            \underline{U}^{i-1} : 11 \rightarrow \hspace{8.10pt}\underline{V}^i : {00}, \ \underline{U}^i : 11\\
        \end{cases}\\
        \underline{V}^{i-1} : 01, \ \hspace{0.85pt}\begin{cases}
            \underline{U}^{i-1} : 01 \rightarrow \hspace{8.10pt}\underline{V}^i : 10, \ \underline{U}^i : 10 \text{ (left jump in both)}\\
            \underline{U}^{i-1} : 11 \rightarrow \hspace{8.10pt}\underline{V}^i : 10, \ \underline{U}^i : 11 \text{ (left jump in $\underline{V}$)}\\
        \end{cases}\\
        \underline{V}^{i-1} : 10, \ \begin{cases}
            \underline{U}^{i-1} : 10 \rightarrow \begin{cases}
                \underline{V}^i : 01, \ \underline{U}^i : 01 \text{ with probability $t$ (right jump in both)}\!\!\!\\
                \underline{V}^i : 10, \ \underline{U}^i : 10 \text{ with probability $1-t$}
            \end{cases}\\
            \underline{U}^{i-1} : 11 \rightarrow \begin{cases}
                \underline{V}^i : 01, \ \underline{U}^i : 11 \text{ with probability $t$ (right jump in $\underline{V}$)}\\
                \underline{V}^i : 10, \ \underline{U}^i : 11 \text{ with probability $1-t$}
            \end{cases}\\
        \end{cases}\\
        \underline{V}^{i-1} : 11, \ \hspace{8.9pt}\underline{U}^{i-1} : 11 \rightarrow \hspace{8.10pt}\underline{V}^i : 11, \ \underline{U}^i : 11
    \end{cases}
    \end{align*}

    This coupling does indeed keep the invariant $\underline{U}^i \ge \underline{V}^i$, and one can easily see it is consistent with the marginal distributions of $\underline{V}$ and $\underline{U}$.
    
    Let $R^i$ be the index of the rightmost 1 in $\underline{U}^i$, or $\infty$ if there is no such index. $X^0(\sigma_i)$ is the index of the rightmost 1 in $\underline{V}^i$, hence $R^i \ge X^0(\sigma_i)$ for all $i$. For all $i$, $n \ge 0$:
    \begin{multline*}
        \bbP\{ R^i \ge n \} = \bbP\{ U^i_j=1 \text{ for some } j \ge n \}\\
	\le \sum_{j=n}^\infty \bbP\{ U^i_j=1 \} = \sum_{j=n}^\infty \frac{t^{j+c}}{1+t^{j+c}} < \sum_{j=n}^\infty t^{j+c} = \frac{t^{n+c}}{1-t}.
    \end{multline*}
    Note that for any $\sigma$ there is an index $i$ such that $X^0(\sigma) = X^0(\sigma_i)$. Hence, at all times $\sigma$ for all $n \ge 0$ we have $\bbP\{ X^0(\sigma) \ge n \} < \frac{t^{n+c}}{1-t}$, and the exponential bound is attained.
\end{proof}
The next corollary uses the fact that $\frac{\omega}{\alpha}$-particles satisfy the law of large numbers stated in Theorem \ref{th:2nd-class-lln} and move along the direction of the slope $\lambda(\lambda+\varepsilon)(1-t)$. Denote by $Q_y$ for the abscissa of the second class particle $Q$ when it reaches the ordinate $y$ for the first time.
\begin{cor}\label{HSlope}
    The rightmost $\frac{\omega}{\eta}$-particle moves along or above the limiting direction of the slope $\lambda(\lambda+\varepsilon)(1-t)$. That is, if $H^0$ is the rightmost $\frac{\omega}{\eta}$-particle, then 
    \begin{align*}
        \bbP\{ y < \left(\lambda(\lambda+\varepsilon)(1-t)-\delta\right)H^0_y\} \stackrel{y \to \infty}{\to} 0
    \end{align*}
    for every $\delta > 0$.
\end{cor}
\begin{proof}
Fix $\kappa > 0$.  By Lemma \ref{RightTail}, there exists $j \in \bbZ_{>0}$ such that $\bbP\{ X^0(\sigma) > j \} < \kappa/2$ for all $\sigma \in \bbR_{\ge 0}$. This implies $\bbP\{ X^0_y > j \} < \kappa/2$ for all $y \in \bbR_{\ge 0}$, where we write $X^0_y$ for the index of the $\frac{\omega}{\alpha}$-particle that coincides with the $\frac{\omega}{\eta}$-particle at ordinate $y$; we view $X^0_y$ as a right-continuous function of $y$. Let $A^j$ be the $j$-th $\frac{\omega}{\alpha}$-particle. Then, due to the limiting direction of $\frac{\omega}{\alpha}$-particles given by Theorem \ref{th:2nd-class-lln}, there exists $y_0 \in \bbR_{\ge 0}$ such that $\bbP\{ y < \left(\lambda(\lambda+\varepsilon)(1-t)-\delta\right)A^j_y \} < \kappa/2$ for all $y \ge y_0$. Hence, using the union bound, we obtain that 
\begin{align*}
    \bbP\{ y < \left(\lambda(\lambda+\varepsilon)(1-t)-\delta\right)H^0_y\} \le \bbP\{ X^0_y > j \} + \bbP\{ y < \left(\lambda(\lambda+\varepsilon)(1-t)-\delta\right)A^j_y \} < \kappa.
\end{align*}
for $y \ge y_0$. This implies $\bbP\{ y < \left(\lambda(\lambda+\varepsilon)(1-t)-\delta\right)H^0_y\} \stackrel{y \to \infty}{\to} 0$ as claimed. 
\end{proof}

Our eventual goal is to show that the zero-boundary $t$-PNG process in the limit has the same local property that characterizes a stationary process, that is, Poisson processes at the intersections of constant-size intervals and space-time paths. We will demonstrate this convergence to Poisson processes by considering finite families of disjoint non-empty intervals and establishing convergences to independent Poisson random variables for the numbers of particles passing through each interval.

Recall the notation from before Theorem \ref{th:LocConv}. The next lemma completes the intuition seen at the end of Section \ref{sc:resmeth}.

\begin{lem}\label{OmegaConv}
The process of intersection points of $\omega$-lines with the horizontal rays $\bbR_{>0}\times\{y_k\}$ and the vertical rays $\{x_k\}\times\bbR_{>0}$ converges in distribution to a Poisson process of rate $\lambda+\varepsilon$ on the horizontal axis and $\frac{1}{(\lambda+\varepsilon)(1-t)}$ on the vertical axis in the sense that, for arbitrary interval families $(\mathcal{I}^i)_{i \in [M]}$ and $(\mathcal{J}^j)_{j \in [N]}$,
\begin{multline}\label{omega_ineq}
    \left( \left|\omega \cap \mathcal{I}^i_k\right| \right)_{i \in [M]} \otimes \left( \left|\omega \cap \mathcal{J}^j_k\right| \right)_{j \in [N]}\\
    \xrightarrow{d} \left(\bigotimes\limits_{i \in [M]} \Pois\left[(\lambda+\varepsilon)m_i\right]\right) \otimes \left(\bigotimes\limits_{j \in [N]} \Pois\left[\frac{n_j}{(\lambda+\varepsilon)(1-t)}\right]\right)
\end{multline}
as $k \to \infty$.
\end{lem}
\begin{proof}
    Due to $\frac{y_k}{x_k} \to \lambda^2(1-t)$, all intervals from $(\mathcal{I}^i)_{i \in [M]}$ and $(\mathcal{J}^j)_{j \in [N]}$ are entirely below the line of the slope $\lambda(\lambda+\varepsilon)(1-t)$ for $k$ large enough. Hence, Corollary \ref{HSlope} implies that $\left( \left|\frac{\omega}{\eta} \cap \mathcal{I}^i_k\right| \right)_{i \in [M]} \otimes \left( \left|\frac{\omega}{\eta} \cap \mathcal{J}^j_k\right| \right)_{j \in [N]} \xrightarrow{p} \vec{0}$.

    Due to stationarity, for every $k$, sets $\eta \cap \mathcal{I}^i_k$ and $\eta \cap \mathcal{J}^j_k$ for all $i$ and $j$ are independent Poisson processes on $\mathcal{I}^i_k$ and $\mathcal{J}^j_k$ of rates $\lambda+\varepsilon$ and $\frac{1}{(\lambda+\varepsilon)(1-t)}$, respectively. Thus, 
    \begin{multline*}
        (|\eta \cap \mathcal{I}^i_k|)_{i \in [M]} \otimes (|\eta \cap \mathcal{J}^j_k|)_{j \in [N]}\\
	\sim \left(\bigotimes\limits_{i \in [M]} \Pois\left[(\lambda+\varepsilon)m_i\right]\right) \otimes \left(\bigotimes\limits_{j \in [N]} \Pois\left[\frac{n_j}{(\lambda+\varepsilon)(1-t)}\right]\right).
    \end{multline*}
    
    Note that $(\omega \cap \mathcal{I}^i_k) \supseteq (\eta \cap \mathcal{I}^i_k)$ for all $i$ and $k$, and therefore $\left|(\omega \cap \mathcal{I}^i_k) \setminus (\eta \cap \mathcal{I}^i_k)\right| = \left|\frac{\omega}{\eta} \cap \mathcal{I}^i_k\right| \xrightarrow{p} 0$. Thus $\left( \left|\omega \cap \mathcal{I}^i_k\right| \right)_{i \in [M]} - \left( \left|\eta \cap \mathcal{I}^i_k\right| \right)_{i \in [M]} \xrightarrow{p} \vec{0}$, and similarly $\left (\left|\eta \cap \mathcal{J}^j_k\right| \right)_{j \in [N]} - \left( \left|\omega \cap \mathcal{J}^j_k\right| \right)_{j \in [N]} \xrightarrow{p} \vec{0}$, which implies the lemma.
\end{proof}

Consider another process $\widehat{\omega} \sim t\text{-PNG}(\widehat{\lambda}+\varepsilon, 0)$ with $\widehat{\lambda} = \frac{1}{\lambda(1-t)}$. Pick a sequence $(\widehat{\bfv}_k = (\widehat{x}_k, \widehat{y}_k))_{k \in \bbZ_{>0}}$ such that $|\widehat{\bfv}_k| \to \infty$ and $\frac{\widehat{y}_k}{\widehat{x}_k} \to \widehat{\lambda}^2(1-t)$, and define the interval systems $(\widehat{\mathcal{I}}^i_k)_{i \in [\widehat{M}], k \in \bbZ_{>0}}$, $(\widehat{\mathcal{J}}^j_k)_{j \in [\widehat{N}], k \in \bbZ_{>0}}$ with lengths $(\widehat{m}_i)_{i \in [\widehat{M}]}$ and $(\widehat{n}_j)_{j \in [\widehat{N}]}$ similarly to above. Then, by Lemma \ref{OmegaConv}, 
\begin{multline*}
    \left( \left|\widehat{\omega} \cap \widehat{\mathcal{I}}^i_k\right| \right)_{i \in [\widehat{M}]} \otimes \left( \left|\widehat{\omega} \cap \widehat{\mathcal{J}}^j_k\right| \right)_{j \in [\widehat{N}]}\\
    \xrightarrow{d} \left(\bigotimes\limits_{i \in [\widehat{M}]} \Pois\left[(\widehat{\lambda}+\varepsilon)\widehat{m}_i\right]\right) \otimes \left(\bigotimes\limits_{j \in [\widehat{N}]} \Pois\left[\frac{\widehat{n}_j}{(\widehat{\lambda}+\varepsilon)(1-t)}\right]\right).
\end{multline*}

Introduce 2 more $t$-PNG processes:
\begin{itemize}
    \item $\chi = \chi(\varepsilon) \sim t\text{-PNG}(0, \widehat{\lambda}+\varepsilon)$;
    \item $\theta \sim t\text{-PNG}(0,0)$, a zero-boundary process.
\end{itemize}

There exists a coupling of $\chi$ and $\widehat{\omega}$ in which $\chi$ is obtained from $\widehat{\omega}$ by swapping the coordinate axes; we can then take
\begin{itemize}
 \item$((x_k, y_k))_{k\in\bbZ_{>0}}$ corresponding to $((\widehat{x}_k, \widehat{y}_k))_{k \in \bbZ_{>0}}$,
 \item$(\mathcal{I}^i_k)_{i \in [M], k \in \bbZ_{>0}}$ to $(\widehat{\mathcal{J}}^j_k)_{j \in [\widehat{N}], k \in \bbZ_{>0}}$,
 \item$(\mathcal{J}^j_k)_{j \in [N], k \in \bbZ_{>0}}$ to $(\widehat{\mathcal{I}}^i_k)_{i \in [\widehat{M}], k \in \bbZ_{>0}}$,
 \item$(m_i)_{i \in [M]}$ to $(\widehat{n}_j)_{j \in [\widehat{N}]}$, and
 \item$(n_j)_{j \in [N]}$ to $(\widehat{m}_i)_{i \in [\widehat{M}]}$
\end{itemize}
 with respect to diagonal reflection. Therefore, 
\begin{multline}\label{chi_ineq}
    \left( \left|\chi \cap \mathcal{I}^i_k\right| \right)_{i \in [M]} \otimes \left( \left|\chi \cap \mathcal{J}^j_k\right| \right)_{j \in [N]}\\
    \xrightarrow{d} \left(\bigotimes\limits_{i \in [M]} \Pois\left[\frac{m_i}{(\widehat{\lambda}+\varepsilon)(1-t)}\right]\right) \otimes \left(\bigotimes\limits_{j \in [N]} \Pois\left[(\widehat{\lambda}+\varepsilon)n_j\right]\right).
\end{multline}

\begin{proof}[Proof of Theorem \ref{th:LocConv}]
    Define all aforementioned constructions with $\lambda = \frac{1}{\sqrt{1-t}}\sqrt{\frac{y}{x}}$. For any $\varepsilon > 0$, Corollary \ref{cor:poi-coupling} allows an increasing coupling $\omega \ge \theta \ge \chi$. Under this coupling, for all $k$, we have 
    \begin{align*}
        (\omega \cap (\bbR_{>0}\times\{y_k\})) \supseteq (\theta \cap (\bbR_{>0}\times\{y_k\})) \supseteq (\chi \cap (\bbR_{>0}\times\{y_k\})). 
    \end{align*}
    Therefore, coordinatewise, 
    \begin{align}\label{hor-ineq}
        \left( \left|\omega \cap \mathcal{I}^i_k\right| \right)_{i \in [M]} \ge \left( \left|\theta \cap \mathcal{I}^i_k\right| \right)_{i \in [M]} \ge \left( \left|\chi \cap \mathcal{I}^i_k\right| \right)_{i \in [M]}.
    \end{align}
    Similarly, 
    \begin{align}\label{vert-ineq}
        \left(\left|\chi \cap \mathcal{J}^j_k\right|\right)_{j \in [N]} \ge \left(\left|\theta \cap \mathcal{J}^j_k\right|\right)_{j \in [N]} \ge \left(\left|\omega \cap \mathcal{J}^j_k\right|\right)_{j \in [N]}.
    \end{align}
    Combining \ref{hor-ineq} and \ref{vert-ineq} gives 
    \begin{align}
    \label{E:1}
    \begin{split}
        &\left( \left|\omega \cap \mathcal{I}^i_k\right| \right)_{i \in [M]} \otimes \left( -\left|\omega \cap \mathcal{J}^j_k\right| \right)_{j \in [N]}\\ 
        \ge &\left( \left|\theta \cap \mathcal{I}^i_k\right| \right)_{i \in [M]} \otimes \left( -\left|\theta \cap \mathcal{J}^j_k\right| \right)_{j \in [N]}\\
        \ge &\left( \left|\chi \cap \mathcal{I}^i_k\right| \right)_{i \in [M]} \otimes \left( -\left|\chi \cap \mathcal{J}^j_k\right| \right)_{j \in [N]}.
    \end{split}
    \end{align}
    Now send $k \to \infty$ and then $\varepsilon \to 0$ in \eqref{E:1}. Then, since the right-hand sides in \ref{omega_ineq} and \ref{chi_ineq} have a common limit as $\varepsilon \to 0$, we obtain that 
    \begin{align*}
        \left( \left|\theta \cap \mathcal{I}^i_k\right| \right)_{i \in [M]} \otimes \left( \left|\theta \cap \mathcal{J}^j_k\right| \right)_{j \in [N]} \xrightarrow{d} \left(\bigotimes\limits_{i \in [M]} \Pois\left[\lambda m_i\right]\right) \otimes \left(\bigotimes\limits_{j \in [N]} \Pois\left[\frac{n_j}{\lambda(1-t)}\right]\right)
    \end{align*}
    as $k \to \infty$.
\end{proof}

\bibliographystyle{habbrv}
\bibliography{Refs}

\end{document}